\documentclass[12pt]{article}
\usepackage{amssymb,amsthm,amsmath,pb-diagram}
\usepackage[OT2,T1]{fontenc}
\usepackage[all]{xy}
\usepackage{url}
\usepackage{mathrsfs}
\DeclareSymbolFont{cyrletters}{OT2}{wncyr}{m}{n}
\DeclareMathSymbol{\Sha}{\mathalpha}{cyrletters}{"58}
\usepackage{fullpage}
\newtheorem{thm}[equation]{Theorem}
\newtheorem*{thm*}{Theorem}
\newtheorem{lemma}[equation]{Lemma}
\newtheorem{prop}[equation]{Proposition}
\newtheorem{cor}[equation]{Corollary}
\theoremstyle{definition}
\newtheorem{remark}[equation]{Remark}
\newtheorem{notation}[equation]{Notation}
\newtheorem{dfn}[equation]{Definition}
\newtheorem{example}[equation]{Example}
\newtheorem{hyp}[equation]{Hypothesis}

\newtheorem*{guidquest*}{Guiding Question}
\newcommand{\m}[1]{#1}
\newcommand{\mbb}[1]{\mathbb #1}
\newcommand{\mc}[1]{\mathcal #1}
\newcommand{\oper}[1]{\operatorname{#1}}
\newcommand{\wh}{\widehat}

\newcommand{\GL}{\oper{GL}}
\newcommand{\FS}{F^{\operatorname{split}}}
\newcommand{\PP}{\oper{PP}}

\newcommand{\Hom}{\oper{Hom}}
\newcommand{\Epi}{\oper{Epi}}
\newcommand{\Vect}{\oper{Vect}}
\newcommand{\Spec}{\oper{Spec}}

\newcommand{\Gal}{\oper{Gal}}

\newcommand{\cha}{\oper{char}}
\newcommand{\id}{\oper {id}}
\newcommand{\Mat}{\oper {Mat}}
\newcommand{\sep}{\mathrm{sep}}
\newcommand{\iso}{\to^{\!\!\!\!\!\!\!\sim\,}}
\renewcommand{\labelenumi}{(\theenumi)}
\newcommand{\ind}{\oper{ind}}

\newcommand{\PGL}{\oper{PGL}}
\newcommand{\PGU}{\oper{PGU}}
\newcommand{\GU}{\oper{GU}}
\newcommand{\SB}{\oper{SB}}
\renewcommand{\U}{\oper{U}}
\newcommand{\Aut}{\oper{Aut}}
\def\<{\left<}
\def\>{\right>}
\newcommand{\dra}{\dashrightarrow}
\renewcommand{\split}{\mathrm{split}}
\newcommand{\s}{\split}
\newcommand{\Br}{\oper{Br}}
\newcommand{\Orth}{\oper{O}}
\newcommand{\SOrth}{\oper{SO}}
\renewcommand{\theenumi}{\alph{enumi}}
\renewcommand{\labelenumi}{(\alph{enumi})}
\newcommand{\T}{\mathcal{T}}

\title{Local-global principles for torsors over arithmetic curves}
\author{David Harbater, Julia Hartmann, and Daniel Krashen}
\date{}
\numberwithin{equation}{section}

\begin{document}
\maketitle

\begin{abstract}\let\thefootnote\relax
We consider local-global principles for torsors under linear algebraic groups,
over function fields of curves over complete discretely valued fields. 
The obstruction to such a principle is a version of the Tate-Shafarevich
group; and for groups with rational components, we compute it explicitly and show that it is finite.  This yields
necessary and sufficient conditions for local-global
principles to hold.  Our results rely on first obtaining a Mayer-Vietoris sequence for Galois cohomology and then showing that torsors can be patched.  We also give
new applications to quadratic forms and central simple algebras. 
\end{abstract}

\section{Introduction}\label{intro}

Classical local-global principles study objects defined over a global
field $F$ by relating them to objects defined over the completions $F_v$ of the field.  
For example, the Hasse-Minkowski theorem states that a quadratic form over $F$ is isotropic if and only if it is isotropic over each completion.  
Similarly, the Albert-Brauer-Hasse-Noether theorem asserts
that a central simple algebra over $F$ is split if and only if it is split over each completion.  

Local-global principles can often be phrased in terms of torsors (principal homogeneous spaces) under group schemes $G$ over $F$, which are classified by $H^1(F,G)$ in Galois cohomology.  An advantage of this point of view 
is that one can then additionally consider the obstruction to the local-global principle, viz.\ the kernel $\Sha(F,G)$ of the local-global map $H^1(F,G) \to \prod H^1(F_v,G)$.  According to the conjecture of Birch and Swinnerton-Dyer, if $E$ is an elliptic curve over a global field $F$, 
the Tate-Shafarevich group $\Sha(F,E)$ has a specified finite order, which can be greater than one.  In the case of torsors under a linear algebraic group $G$, the local-global obstruction $\Sha(F,G)$ again need not vanish, but 
it is known to be finite.  For number fields this is due to Borel-Serre (\cite{BS}), and for function fields it was shown by Conrad (\cite{Conrad}) following earlier work in \cite{Oesterle} and \cite{BP}.

In recent years, local-global principles have also been studied for objects defined over ``higher dimensional'' fields, such as one-variable function fields over local or global fields, or completions of such fields.  
This work has been carried out by Kato, Colliot-Th\'el\`ene, Parimala, and others; e.g.\ see
\cite{COP}, \cite{CGP}, \cite{BKG}, and \cite{Kato}.  In \cite{HHK}, the present authors obtained local-global principles for one-variable function fields $F$ over arbitrary complete discretely valued fields $K$, and used them to obtain applications to quadratic forms and central simple algebras. 
Those principles, which concerned rational connected linear algebraic groups over $F$, were stated with respect to a finite set of overfields of $F$ arising from the patching framework of \cite{HH:FP}, rather than an infinite set arising from completions at discrete valuations.  Afterwards, in \cite{CPS}, it was shown that in the specific cases of quadratic forms and central simple algebras, such local-global principles also hold with respect to the set of discrete valuations.  

These last two papers, however, had limitations that provided motivation for the present manuscript.  In \cite{HHK}, the group was required to be connected.  
Although this condition was not essential in the case $F=K(x)$ (or any function field of good reduction), in the general case connectivity is needed in order to obtain local-global principles,  
even under a rationality assumption, as Colliot-Th\'el\`ene observed
(see \cite[Remark~4.4]{CPS}, \cite[Example~4.4]{HHK}).
As a consequence, there are counterexamples to the local-global principle for binary quadratic forms over such function fields, since the orthogonal group has two components (both of which are rational).  
Also, while \cite{CPS} obtains local-global principles with respect to discrete valuations in the case of quadratic forms and central simple algebras, the analogous assertion for more general linear algebraic groups (even in the rational connected case) was left open.  We are therefore led to the following

\begin{guidquest*} \label{quest}
For a one-variable function field $F$ over a complete discretely valued field $K$, do 
local-global principles hold for torsors under linear algebraic groups that are not necessarily rational and connected?  If not, what is the obstruction and is it finite?
\end{guidquest*}

In the present manuscript, we consider obstructions to local-global principles for linear algebraic $F$-groups $G$ that need not be connected, but which are {\em rational} in the sense that every connected component is a rational $F$-variety.  
(This includes the important case of orthogonal groups.)
We give a necessary and sufficient condition for such groups to satisfy these local-global principles with respect to a finite set of overfields; 
and we compute the precise obstruction, which turns out to be finite but not necessarily trivial.  The obstruction, which is given in terms of the reduction graph $\Gamma$ of a normal projective model $\wh X$ of $F$ over the valuation ring $T$ of $K$, explains the results in \cite{HHK} as well as the counterexample of Colliot-Th\'el\`ene.
Moreover this obstruction agrees with the
obstruction $\Sha_X(F,G) = \ker\bigl(H^1(F,G) \to \prod H^1(F_P,G)\bigr)$ to the local-global principle with respect to the infinite set of overfields $F_P$ that arise from completions at the points $P$ of the closed fiber $X$ of $\wh X$.  We show:

\begin{thm*} [Corollary~\ref{Sha_P-graph}] \label{intro_Sha_0}
For $G$ rational, the set $\Sha_X(F,G)$ is finite, and it is in natural bijection with the set $\Hom(\pi_1(\Gamma),G/G^0)$.  Thus the corresponding local-global principle holds (i.e.\ $\Sha_X(F,G)$ is trivial) if and only if $G$ is connected or $\Gamma$ is a tree.
\end{thm*}

\noindent Other characterizations are given in terms of split covers and split extensions
(Theorems~\ref{sha_fsplit} and~\ref{sha_fs}),
implying that the obstruction is independent of the choice of a regular model. 

We also address the question of whether local-global principles hold with respect to discrete valuations for fields $F$ as above.  This question has been a focus of recent interest (by Parimala and others) especially in the connected case, and it remains open in general.  We show here that for any linear algebraic group $G$ (even if disconnected), the obstruction $\Sha(F,G)$ to a local-global principle with respect to discrete valuations contains the above obstruction $\Sha_X(F,G)$ (Proposition~\ref{sha_exact_sequence}), thereby providing a necessary condition for a local-global principle in terms of discrete valuations if $G$ is rational.  Moreover, under certain extra hypotheses we are able to show that the inclusion $\Sha_X(F,G) \subseteq \Sha(F,G)$ is actually an equality.  For example, we have:

\begin{thm*} [Theorem~\ref{sha_conditions}, Corollary~\ref{Sha_descrip}]
\label{intro_Sha}
Suppose that the residue field of $T$ is algebraically closed of characteristic zero and that $G$ is rational.  Then $\Sha(F,G) = \Sha_X(F,G)$.  Hence the local-global principle with respect to discrete valuations then holds if and only if $G$ is connected or $\Gamma$ is a tree.
\end{thm*}

Using the above results, we obtain applications to quadratic forms and central simple algebras over $F$.  These include local-global principles for the Witt index of a quadratic form and for the Witt group of quadratic form classes (Section~\ref{subsec_quadforms}); and for the index of a central simple algebra and for whether such an algebra is split (Section~\ref{subsec_csa}).  Using the framework developed here, further applications are obtained in two sequels to this paper.  These include values of $u$-invariants and period-index bounds, in \cite{HHK:Weier}; and local-global principles for higher Galois cohomology, in \cite{HHK:Hi}.

As for local-global principles holding for groups that are not known to be rational, 
results of this sort are obtained in Section~4.3 of \cite{HHK:Hi} by applying
cohomological invariants to the higher cohomology results there. 
Also, connected linear algebraic groups that are merely {\em retract rational} satisfy such local-global principles (see \cite{Kra}).
But recently, in \cite{CPS13}, an example was given of a 
non-rational connected linear algebraic group over $F$ for which local-global principles do not hold.  For now the precise obstruction in the non-rational case, as well as its finiteness, remains open.

\bigskip

\noindent \textbf{Methods and structure of the manuscript.}  
Our approach relies heavily on the patching framework begun in~\cite{HH:FP} and expanded on in~\cite{HHK}.  Here we develop this further, in order to patch torsors and to obtain a ``Mayer-Vietoris'' exact sequence that in particular relates local-global principles to matrix factorization.

Given a normal projective model $\wh X$ of $F$ and a sufficiently large finite set of points $\mc P$ on the closed fiber $X$, we consider the set $\mc U$ of connected components $U$ of $X \smallsetminus \mc P$, and the set $\mc B$ of branches $\wp$ of $X$ at the points $P \in \mc P$.  We then 
obtain associated fields $F_U$ and $F_P$ and overfields $F_\wp$, which together form a connected simplicial set.  
This is unlike the situation for completions with respect to discrete valuations, or with respect to the points on the closed fiber, where distinct completions do not have natural common overfields.  In fact, the presence of the common overfields $F_\wp$ is what enables us to obtain our results.  

As shown in~\cite[Theorem~6.4]{HH:FP}, finite dimensional vector spaces (possibly with additional structure) may be patched in this setup, in the sense
that if objects are given over all the fields $F_U$ and $F_P$, together with isomorphisms over the fields $F_\wp$, then there is a unique object over $F$ that induces all of them compatibly.
This assertion can be viewed as a birational analog of Grothendieck's Existence Theorem (\cite[Corollary 5.1.6]{EGA3}), but it is better adapted to our situation here, where we work with structures over function fields, rather than over the underlying schemes.  
Note that both~\cite[Theorem~6.4]{HH:FP} and  
\cite[Corollary 5.1.6]{EGA3} require that the objects under consideration are finite over the base.  Nevertheless, we show in Theorem~\ref{patching_torsors} that patching holds for $G$-torsors over $F$, even if the group $G$ is infinite (in which case the coordinate ring of the torsor is an infinite dimensional $F$-vector space).  Building on this, we obtain our six-term Mayer-Vietoris sequence in Galois cohomology, which allows us to treat patches as though they corresponded to classical open sets.  (In the sequel \cite{HHK:Hi}, we show for certain types of commutative groups $G$ that this sequence can be continued into higher Galois cohomology.)

These results are shown in an abstract context in Section~\ref{patching}, and then for our function fields $F$ in Section~\ref{setup}.  The main technical part of the paper appears in Section~\ref{sec_rational}, where for $G$ rational
the study of the kernel $\Sha_{\mc P}(F,G)$ of
$H^1(F,G) \to \prod_{\xi \in \mc P \cup \mc U} H^1(F_\xi,G)$
is reduced to the case of finite groups via an analytic argument that draws on~\cite[Section~2]{HHK}.  
Section~\ref{sec_splitcovers} introduces $\Sha_X(F,G)$ and shows that it agrees with $\Sha_{\mc P}(F,G)$ for rational groups, using the Artin Approximation Theorem.  It then gives a description of $\Sha_X(F,G)$ in terms of $\pi_1^\s(\wh X)$, the quotient of the \'etale fundamental group $\pi_1(\wh X)$ that classifies split covers of $\wh X$ (i.e.\ covers that are split over every point other than the generic point).
In Section~\ref{sec_reduction}, the reduction graph $\Gamma$ of a model $\wh X$ is used to show that $\Sha_X(F,G)$ is finite and can be described explicitly for rational groups; this description
shows that $\Sha_X(F,G)$ is trivial if and only if $G$ is connected or $\Gamma$ is a tree.
Then in Section~\ref{sec_splitexten}, it is shown for $G$ rational that $\Sha_X(F,G)$ is independent of the choice of a regular model $\wh X$ of $F$, as a consequence of proving that $\Sha_X(F,G)=H^1(\FS/F,G)$, where $\FS$ is the maximal algebraic extension of $F$ that splits over every discrete valuation.  (Note that $\FS$ does not have a non-trivial analog in the case of global fields, because of the Chebotarev Density Theorem.)
These results make it possible in 
Section~\ref{sec_valuations} to analyze
the relationship between $\Sha(F,G)$ and $\Sha_X(F,G)$, showing in several situations that they are equal.  As a consequence, we can then obtain results in Section~\ref{sec_applications} that are analogous to classical local-global results over global fields with respect to discrete valuations.
In the case of quadratic forms, where 
Colliot-Th\'el\`ene's original example occurred, we prove local-global results concerning the Witt index and the Witt group, describing the obstruction explicitly.  We also prove local-global principles for central simple algebras concerning the index and splitting.
In that last section we also consider homogeneous spaces that are not necessarily torsors.  

\bigskip

\noindent \textbf{Terminology and notions.}

In this manuscript, discrete valuations are assumed to have value group isomorphic to $\mbb Z$.  In particular, they cannot be trivial.  Each codimension one point
on a connected Noetherian normal scheme defines a discrete valuation on the function field, although not every discrete valuation on the function field need arise in this way.

Given a discrete valuation ring $T$ with fraction field $K$, and a field $F$ of finite type and transcendence degree one over $K$, a {\em normal} (resp.\ {\em regular}) {\em model} of $F$ over $T$ 
is a normal (resp.\ regular) connected projective $T$-curve $\wh X$ with function field $F$.  
The closed fiber $X \subset \wh X$ contains all the closed points of $\wh X$, along with the generic points of the irreducible components of $X$, which are of codimension one in $\wh X$.  The other codimension one points of $\wh X$ are the closed points of the generic fiber~$\wh X_K$.

For any field $K$, we write $\Gal(K)$ for the absolute Galois group $\Gal(K^\sep/K)$. For any $K$-scheme $V$ and any field extension $L/K$, we write $V_L$ for the base change $V \times_K L$.  

All {\em group schemes} are assumed to be of finite type over the base.  By a {\em linear algebraic group} over a field $E$ we mean 
a smooth affine group scheme of finite type (cf.~\cite{Borel}); or equivalently, a geometrically reduced Zariski closed subgroup of some general linear group $\GL_{n,E}$. 
We will consider (right) {\em torsors} over $E$, under linear algebraic groups $G$;
i.e.\ $E$-schemes $H$ together with a right $G$-action such that $(h,g) \mapsto (h, h\cdot g)$ gives an isomorphism $H \times G \to H \times H$.  By the smoothness condition on $G$,
this is equivalent to considering $G$-torsors $H$ for the \'etale topology, or alternatively $G$-torsors $H$ for the flat topology.

We will consider Galois cohomology with respect to a linear algebraic group $G$ that need not be commutative, as in~\cite{Serre:CG}, Chapter~III.  Here $H^1(F,G)$ is just a pointed set, not a group.  The {\em kernel} of a map of pointed sets consists of the elements that map to the identity; and having trivial kernel is in general weaker than being injective.  Exactness is defined as usual in terms of kernel and image, with respect to this notion.

\medskip

\noindent{\bf Acknowledgments.}  The authors thank Jean-Louis
Colliot-Th\'el\`ene for a number of very helpful conversations concerning material in this
manuscript and related ideas.  We also thank Brian Conrad, Ofer Gabber, Philippe Gille, Qing
Liu, George McNinch, R.~Parimala, Florian Pop, and Steve Shatz for
helpful comments and discussions. We are indebted to Annette Maier for her
careful reading of the manuscript.

\section{Patching Problems and Torsors}\label{patching}
In this section we consider patching problems for torsors under linear
algebraic groups.
We show that these patching problems can be solved and that there is an equivalence of
categories (Theorem~\ref{patching_torsors}).  Since the coordinate ring of a
torsor is in general infinite dimensional as a vector space, this result goes
beyond the type of patching theorem obtained in \cite[Section~7]{HH:FP}. Using
these results, we obtain a six term exact sequence at the level
 of $H^0$ and $H^1$ (Theorem~\ref{general_6-term_sequence}).

{\em Patching} considers a field $F$ and a collection of overfields. It asserts that
if compatible structures are given over these overfields, then there is a
structure over $F$ inducing all of them compatibly. We first recall how it can
be described in the situation in which vector spaces are the structure under
consideration.  Let $\mc F := \{F_i\}_{i\in I}$ be a finite inverse system of
fields and inclusions, whose inverse limit (in the category of rings) is a field $F$.

For $i,j \in I$, write $i \succ  j$ if we are given a proper inclusion $F_i
\hookrightarrow F_j$.  A (vector space) \textit{patching problem} for $\mc
F$ is a system $\mathcal{V} := \{V_i\}_{i \in I}$ of finite dimensional
$F_i$-vector spaces for $i \in I$, together with $F_j$-isomorphisms
$\nu_{i,j}:V_i  \otimes_{F_i} F_j  \to V_j$ for all $i\succ j$ in $I$.  (This
is equivalent to the definition in \cite[Section~2]{HH:FP}.)  A morphism of
patching problems is defined in the obvious way.
We write $\PP(\mc F)$ for the
category of vector space patching problems for $\mc F$.

In the above setup, there is a functor
\[\beta: \Vect(F) \to \PP(\mc F)\]
from the category
of finite dimensional $F$-vector spaces to the category of vector space
patching problems for $\mc F$ defined by base change. A \textit{solution} to a patching problem
$\mathcal V$ for $\mc F$ is an $F$-vector space $V$ such that $\beta(V)$ is
isomorphic to $\mathcal V$ in the category of patching problems. (See
\cite[Section~2]{HH:FP} for more details.) If $\beta$ is an equivalence of
categories, then every patching problem has a unique solution up to (a unique) isomorphism.

The simplest type of patching problem consists of two fields $F_1,F_2$ with  a common overfield $F_0$ and intersection $F$. In this situation, $\beta$
is an equivalence of categories if and only if every matrix $A_0 \in
\GL_n(F_0)$ can be factored as $A_1^{-1}A_2$ with $A_i \in \GL_n(F_i)$.  (See
\cite{HH:FP}, Section~2.)

In this paper, we are concerned with a special type of inverse system.

\begin{dfn}\label{factorsystem}
A {\em factorization inverse
  system} over a field $F$ is a finite inverse system of fields whose inverse limit (in the category of rings) is $F$, and whose index set $I$ has the following property: There is a
partition $I=I_v\cup I_e$ into a disjoint union such that for each index $k\in
I_e$, there are exactly two elements $i,j \in I_v$ for   which $i,j\succ k$,
and there are no other relations in $I$.
\end{dfn}

A factorization inverse system determines a (multi-)graph $\Gamma$ whose vertices are
the elements of $I_v$ and whose edges are the elements of $I_e$.  The
vertices of an edge $k \in I_e$ correspond to the elements $i,j \in I_v$
such that $i,j \succ k$.  There are also associated fields $F_i, F_k$ ($i \in I_v$, $k \in I_e$) and inclusions $F_i \hookrightarrow F_k$ for $i$ a vertex of an edge $k$.  
Here $\Gamma$ is connected because the inverse limit $F$ has no zero divisors.  
Conversely, a connected graph $\Gamma$ together with fields and inclusions (called 
a {\em $\Gamma$-field} in \cite[Section~2.1]{HHK:Hi})
determines a factorization inverse system.
  
Notice that the simple inverse system consisting of two fields $F_1,F_2$ with a
common overfield $F_0$ is a factorization inverse system. In fact, our
interest in factorization inverse systems is based on a generalization of the
factorization property for $\GL_n(F_0)$ mentioned above. Namely, if
$\{F_i\}_{i\in I}$ is a factorization inverse system whose inverse limit is a field $F$, then the base change
functor $\beta$ is an equivalence of categories if and only if the {\em
  simultaneous   factorization} property holds, as we explain next.

Fix once and for all for each index $k\in I_e$ a labeling $l=l_k$, $r=r_k$ of the two
elements of $I_v$ that dominate $k$ in the inverse system (this corresponds to
orienting the edges of the graph, i.e., assigning to each edge a left and
right vertex), thereby associating to each $k\in I_e$ an ordered triple
$(l,r,k)$.  The oriented system is then determined by the set $S_I$ of such
triples.

Note that giving a patching problem for $\mc F$ is then equivalent to giving a
collection of $F_i$-vector spaces $V_i$ for $i\in I_v$ together with
isomorphisms $\mu_k:V_l \otimes_{F_l} F_k \to V_r \otimes_{F_r}
F_k$, where $(l,r,k)$ ranges over the set $S_I$.  In the notation above,
$\mu_k = \nu_{r,k}^{-1}\nu_{l,k}$.  (Compare the discussion at the end of
Section~2 of \cite{HH:FP}.)

In the above situation, if $G$ is a linear algebraic group over $F$, then 
we say that \textit{simultaneous factorization holds
for $G$ over} $\mc F$ 
if for any collection of elements $A_k \in
G(F_k)$, for $k\in I_e$, there exist elements $A_i\in G(F_i)$ for all $i\in I_v$ such that $A_k=A_r^{-1}A_l\in G(F_k)$ for all $(l,r,k)\in S_I$, with respect
to the inclusions $F_l,F_r \hookrightarrow F_k$. Notice that an index $i\in
I_v$ may dominate several indices $k \in I_e$, but the matrix $A_i$ is the
same for all corresponding inclusions (hence the use of the term {\em
  simultaneous}).  If simultaneous factorization holds for $\GL_n$ over $\mc F$ for every $n\ge 1$, then we simply say that \textit{simultaneous factorization holds
over}~$\mc F$.
Also note that whether simultaneous factorization holds over $\mc F$ is independent of the chosen left-right labels (since if the labels are reversed for $k \in I_e$, we can replace $A_k$ by its inverse).

We now have the following result, which in the 
case of the simple factorization system $F_1,F_2 \subseteq F_0$ was shown
in Proposition~2.1 of \cite{Ha:CAPS} (see also Proposition~2.1 of
\cite{HH:FP}).

\begin{prop} \label{factn_equiv_of_cats}
Let $\mc F$ be a factorization inverse system over a field $F$.
Then the functor $\beta: \Vect(F) \to \PP(\mc F)$ is an
equivalence of categories if and only if simultaneous factorization holds over
$\mc F$.
\end{prop}

\begin{proof}
First suppose that $\beta$ is an equivalence of categories, and consider a
collection of matrices $A_k \in \GL_n(F_k)$ for $k\in I_e$.  For each $i \in
I_v$ choose an $n$-dimensional $F_i$-vector space $V_i$ with basis $B_i$.  For
each triple $(l,r,k) \in S_I$, and with respect to the bases $B_l,B_r$, the
matrix $A_k$ defines a vector space isomorphism $\mu_k:V_l \otimes_{F_l} F_k
\to V_r \otimes_{F_r} F_k$. These isomorphisms define a patching problem for
$\mc F$, which then has a solution $V$ over $F$.  We thus have isomorphisms
$\alpha_i:V \otimes_F F_i \to V_i$ that are compatible with the maps
$\mu_k$. Choose a basis $B$ of $V$, and let $A_i \in \GL_n(F_i)$ be the matrix
corresponding to $\alpha_i$ with respect to $B,B_i$.  Then $A_k=A_r^{-1}A_l$
for each $(l,r,k) \in S_I$, proving simultaneous factorization.

Concerning the converse, first note that the functor $\beta$ is always fully faithful, even without assuming simultaneous factorization.  Namely, it suffices to check this in the case of morphisms of one-dimensional vector spaces, i.e.\ for $\Hom(F,F)$, and there it follows from the equality $\displaystyle F = \lim_\leftarrow F_i$.

So to prove the converse, we now assume that simultaneous factorization holds, and prove essential surjectivity of $\beta$.  
Consider a
patching problem $\mathcal{V} := \{V_i\}_{i \in I}$ for $\mc F$, corresponding
to isomorphisms $\mu_k$ as above. Choose bases $B_i$ for each $V_i$, and let
$A_k$ be the transition matrix between $V_r$ and $V_l$ for each triple
$(l,r,k) \in S_I$.  By hypothesis, there exist matrices $A_i \in \GL_n(F_i)$
for all $i \in I_v$ such that $A_k=A_r^{-1}A_l$ for each $(l,r,k) \in
S_I$. Adjusting the bases $B_i$ by the matrices $A_i$, we obtain new bases
$B_i'$ for $V_i$ such that $B_l',B_r'$ have the same images in $V_k$ for each
$(l,r,k) \in S_I$.  This yields a solution to the patching problem, given by
an $n$-dimensional $F$-vector space $V$ with a basis $B$, together with the
isomorphisms $V\otimes_F F_i \to V_i$ that take $B$ to $B_i'$. Thus every patching problem
has a solution, and so $\beta$ is surjective on isomorphism classes of
objects.   \end{proof}

We next consider a different category of patching problems. Let $G$ be a linear
algebraic group over $F$.  
Given a finite inverse system of fields
$\mc F$ with inverse limit $F$, we define a \textit{$G$-torsor patching problem} for
$\mc F$ to consist of a system of $G_{F_i}$-torsors $\T_i$ together with
$F_j$-isomorphisms of $G_{F_j}$-torsors $\nu_{i,j}:\T_i \times_{F_i} F_j \to
\T_j$ for all pairs $i \succ j$ (i.e.\ such that $F_i \subset F_j$).                                                                                      
A $G$-torsor $\T$ over $F$ induces a $G$-torsor patching problem by base
change; and a \textit{solution} to a given $G$-torsor patching problem
consists of a $G$-torsor $\T$ over $F$ that induces the given patching
problem up to isomorphism. Again, the $G$-torsor patching problems form a
category with the obvious definition of a morphism.

If $\mc F$ is a factorization inverse system, we have pairs of inclusions
$F_l\subseteq F_k$ and $F_r \subseteq F_k$ for $(l,r,k) \in S_I$. Giving a
$G$-torsor patching problem for
$\mc F$ is therefore equivalent to giving a
$G_{F_i}$-torsor $\T_i$ for each $i\in I$ together with $G_{F_k}$-torsor isomorphisms $\mu_k:\T_l\times_{F_l}F_k
\to \T_r\times_{F_r}F_k$ for each
$(l,r,k) \in S_I$.  Again, $\mu_k = \nu_{r,k}^{-1}\nu_{l,k}$.

Recall that $G$-torsors over a field are classified by its first Galois cohomology set.
Regard $G$ as a closed subgroup of $\GL_n$, and consider an extension field
$E/F$.  Now $H^1(E,\GL_n)$ is trivial by Hilbert's Theorem~90 (\cite[Theorem~29.2]{BofInv}).  Hence by
Corollaire~1 of~\cite[I.5.4, Proposition~36]{Serre:CG}, we obtain a bijection of pointed sets
\[\GL_n(E)\backslash H^0(E,\GL_n/G) \iso H^1(E,G).\]
Thus $G$-torsors over $E$ are defined by $\Gal(E^{\sep}/E)$-invariant translates $hG_{E^{\sep}}$ of
$G_{E^{\sep}}$, with $h \in \GL_n(E^\sep)$.

Recall also that the coset space $\GL_n/G$ carries the structure of a quasi-projective variety together with
a quotient morphism $\pi:\GL_n \to \GL_n/G$.  For every field extension $E/F$,
and every $h \in \GL_n(E^\sep)$, the fiber of $\pi$ over $\pi(h)$ is the
translate $hG_{E^{\sep}} \subseteq \GL_{n,E^{\sep}}$; and this fiber is $\Gal(E^{\sep}/E)$-invariant if and
only if $\pi(h)$ is.  That is, $hG_{E^{\sep}}$ defines an element of
$H^0(E,\GL_n/G)$, or equivalently a $G$-torsor over $E$, if and only if $\pi(h)$ is defined over $E$.  For details see
\cite{Borel}, Theorem~II.6.8 and its proof.

Using the above, we now prove that patching for vector spaces implies patching
for torsors.

\begin{thm} \label{patching_torsors} 
Let $\mc F:=\{F_i\}_{i\in I}$ be a factorization inverse system over a field $F$, and let $G$ be a linear algebraic group over $F$. If the base
change functor $\beta: \Vect(F) \to \PP(\mc F)$ is an equivalence of categories,
then so is the functor from the category of $G$-torsors over $F$ to the
category of $G$-torsor patching problems for ${\mathcal F}$. In particular,
every $G$-torsor patching problem has a solution that is unique up to isomorphism.  Moreover, on the
level of coordinate rings, this solution is given by taking the inverse limit.
\end{thm}

\begin{proof}
View $G$ as a closed subgroup of $\GL_{n,F}$.  
We first show that the functor is essentially surjective; i.e.\ that
every $G$-torsor patching problem has a solution. 
By the above discussion, we may
assume that the given patching problem consists of $G$-torsors
$\T_i=h_iG_{F_i^{\sep}}$ which are $\Gal(F_i^{\sep}/F_i)$-invariant orbits of
elements $h_i\in \GL_n(F_i^{\sep})$, for $i \in I_v$, together with isomorphisms
$\mu_k:\T_l\times_{F_l}F_k \to \T_r\times_{F_r}F_k$ for each triple $(l,r,k)\in S_I$.

For each $(l,r,k)\in S_I$, the fields $F_l$ and $F_r$ include into $F_k$; thus
the elements $h_l,h_r$ may be viewed as elements of $\GL_n(F_k^{\sep})$.
Consider the collection of elements $g_k:=\mu_k(h_l)h_l^{-1}\in
\GL_n(F_k^{\sep})$. Left multiplication by $g_k$ determines a morphism of
torsors $\lambda_k: h_lG_{F_k^{\sep}}\to g_kh_lG_{F_k^{\sep}} = h_rG_{F_k^{\sep}}$ (a priori defined over $F_k^{\sep}$)
which sends $h_l\in \T_l(F_k^{\sep})$ to $\mu_k(h_l)\in \T_r(F_k^{\sep})$ by definition of $g_k$. Since a morphism of
torsors is determined by the image of one point, $\lambda_k=\mu_k \otimes_{F_k} F_k^{\sep}$. In particular, $\lambda_k$ is also defined over $F_k$
(i.e.\ it commutes with the action of $\Gal(F_k^{\sep}/F_k)$), and hence $g_k \in \GL_n(F_k)$.

By the hypothesis on vector space patching problems for $\mc F$,
Proposition~\ref{factn_equiv_of_cats} implies that there is a collection of
elements $g_i \in \GL_n(F_i)$ for $i\in I_v$ such that $g_k = g_r^{-1}g_l \in
\GL_n(F_k)$ for each triple $(l,r,k) \in S_I$. For $i\in I_v$, let $h_i' = g_i
h_i$ and let $\lambda_i:h_i G_{F_i^{\sep}} \to h_i' G_{F_i^{\sep}}$ be left
multiplication by $g_i$.   Since $g_k =
g_r^{-1}g_l$, the morphisms
$(\lambda_r)_{F_k}^{-1}$ and $ \mu_k \circ (\lambda_l)_{F_k}^{-1}: h_l' G_{F_k^{\sep}}=h_r'G_{F_k^{\sep}}
\to h_r G_{F_k^{\sep}}$ are the same.

For $(l,r,k)\in S_I$, write $h_k'$ for the image of $h_r'$ in $\GL_n(F_k^{\sep})$.  The
translates $h_i' G_{F_i^{\sep}}$ define $G$-torsors over $F_i$; and so as noted before the theorem, the image $Q_i = \pi(h_i')$ in $\GL_n/G$
is defined over $F_i$.
Thus the points $Q_i$, for $i \in I$, form a compatible
system of $F_i$-points of the $F$-variety $\GL_n/G$.  
Since $F$ is the inverse limit of the fields $F_i$, and since an $E$-point on an $F$-variety $Z$ is defined by an $F$-morphism $\kappa(x) \to E$ for some point $x \in Z$, it follows that the points $Q_i$
define a common $F$-point $Q \in (\GL_n/G)(F)$.  The fiber over this $F$-point is a
$\Gal(F^{\sep}/F)$-invariant translate $hG_{F^{\sep}}$ of $G_{F^{\sep}}$ for some $h \in \GL_n(F^\sep)$, such
that $hG_{F_i^\sep} = h_i' G_{F_i^\sep}$ for each $i \in I$.

Consider the $G$-torsor $\T = hG_{F^{\sep}}$ defined over $F$.
Since $(\lambda_r)_{F_k}^{-1}=\mu_k \circ (\lambda_l)_{F_k}^{-1}$ for each
$(l,r,k) \in S_I$, the torsor $\T$, together with the maps $\lambda_i^{-1}:hG_{F_i^\sep} \rightarrow h_i G_{F_i^\sep}$,
defines a solution to the given patching problem.
This completes the proof of essential surjectivity.

Next, we show the last assertion of the theorem.
Let $\T = \Spec(A)$ be a
$G$-torsor over $F$, inducing $\T_i = \Spec(A_i)$ over $F_i$ for $i \in
I$. Consider  the short exact sequence $0\to F \to \prod_{i \in I_v} F_i \to
\prod_{k \in I_e} F_k$, where the $F_k$-entry of the image of $(a_i)$ is
$a_l-a_r$, if $(l,r,k) \in S_I$ is the triple containing $k$. Tensoring with
$A$ shows that $0 \to A \to \prod_{i \in I_v} A_i \to \prod_{k\in I_e} A_k $
is an exact sequence of $F$-vector spaces.
So $A$ is the inverse limit of the rings $A_i$, as desired.

To conclude the proof of the theorem, we show that the functor is fully faithful, i.e.\ bijective on morphisms between corresponding objects.
Given two $G$-torsors $\T,\T'$ over $F$, write $\T = \Spec(A)$ and $\T' =
\Spec(A')$, and similarly for $\T_i, \T'_i$.  A morphism $\T \to \T'$
corresponds to a homomorphism $A' \to A$, and similarly for $\T_i \to
\T_i'$. Since $A \subseteq A_i$, the map on morphisms is injective.  Since $A$
is the inverse limit of the $A_i$, the map is surjective.
\end{proof}

Since $F$ is the inverse limit of $\mc F$, there is an exact sequence (equalizer diagram)
\[0\to F \to \prod \limits_{i \in I_v} F_i \rightrightarrows \prod
\limits_{k\in I_e} F_k,\]
the double arrows corresponding to the inclusions $F_l,F_r \hookrightarrow
F_k$.  This yields an exact sequence
\[1\to G(F) \to \prod \limits_{i \in I_v} G(F_i) \rightrightarrows \prod
\limits_{k\in I_e} G(F_k),\]
which may be rewritten as an exact sequence of pointed sets
\[1\to G(F)\to \prod \limits_{i \in I_v} G(F_i)\to \prod
\limits_{k\in I_e}G(F_k),\]
where the $F_k$-component of an element in the image of $(g_i)\in \prod_{i \in
  I_v} G(F_i)$ is given by $g_r^{-1}g_l$, if $(l,r,k) \in S_I$.
Equivalently,
\[1\to H^0(F,G)\to \prod \limits_{i \in I_v} H^0(F_i,G)\to \prod
\limits_{k\in I_e}H^0(F_k,G).\]

On the other hand, Theorem~\ref{patching_torsors} yields an exact sequence
\[H^1(F,G) \to \prod \limits_{i \in I_v} H^1(F_i,G) \rightrightarrows \prod
\limits_{k\in I_e} H^1(F_k,G).\]
The next theorem shows that these can indeed be combined to a six-term exact
sequence.

\begin{thm} \label{general_6-term_sequence}
Let $\mc F$ be a factorization inverse system of fields over a field $F$, such that
the base change functor $\beta: \Vect(F) \to \PP(\mc F)$ is an equivalence of
categories. Then for any linear algebraic group $G$ over $F$, we have an exact
sequence of pointed sets
\[\xymatrix{
1 \ar[r] & H^0(F,G) \ar[r] &
\prod_{i \in I_v} H^0(F_i,G) \ar[r] & \prod_{k \in I_e} H^0(F_k,G) \ar@<-2pt>
`d[l]`[lld] [lld] \\
& H^1(F, G) \ar[r] & \prod_{i \in I_v} H^1(F_i,G)\ar@<.5ex>[r] \ar@<-.5ex>[r]
& \prod_{k \in I_e} H^1(F_k,G). }\]
Moreover, two elements $(g_k), (\tilde g_k)\in \prod_{k \in I_e} H^0(F_k,G)$
have the same image under the coboundary map if and only if there are
$(g_i)\in \prod\limits_{i\in I_v} G(F_i)$ such that $g_k=g_r^{-1}\tilde g_k g_l$
for all $(l,r,k)\in S_I$.
\end{thm}

\begin{proof}
As above, we consider $G$ as a closed subgroup of $\GL_n$ for some $n$.  All the maps in the desired sequence were given in the above discussion, except for the coboundary map $\delta$, which we now define.  

Given an element $(g_k)\in
\prod_{k \in I_e} H^0(F_k,G) = \prod_{k \in I_e} G(F_k)$,  
consider the $G$-torsor patching problem with $\T_i$ the trivial $G_{F_i}$-torsor $G_{F_i}$ for $i \in I_v$, and with 
$\mu_k:\T_l \times_{F_l} F_k \to \T_r
\times_{F_r} F_k$ given by left multiplication by $g_k$ on $G_{F_k}$.  By Theorem~\ref{patching_torsors}, there is, up to isomorphism, a unique $G$-torsor $\T$ over $F$ that is a solution to this patching problem; and we take the image of $(g_k)$ under $\delta$ to be the isomorphism class of this torsor.  

Exactness at $H^1(F,G)$ now follows from the fact that an object in the kernel of $H^1(F, G) \to \prod H^1(F_i,G)$ is a $G$-torsor $\T$ that is the solution to a $G$-torsor patching problem with trivial torsors over each $F_i$, or equivalently a $G$-torsor $\T$ in the image of $\delta$. 

For the assertion on the fibers of $\delta$, which includes
exactness at $\prod_{k\in I_e}H^0(F_k,G)=\prod_{k \in I_e} G(F_k)$,
consider two elements $(g_k),(\tilde g_k) \in \prod_{k\in I_e} G(F_k)$.
For $i \in I_v$ let $\T_i$ be the trivial $G_{F_i}$-torsor $G_{F_i}$ over $F_i$; and let $\mu_i$ (resp.\ $\tilde \mu_i$) be left multiplication by $g_i$ (resp.\ $\tilde g_i$).  The tuples $(g_k),(\tilde g_k)$ map to the same
element of $H^1(F,G)$ under $\delta$
if and only if the respective patching problems $(\{\T_i\},\{\mu_k\})$, $(\{
\T_i\},\{\tilde \mu_k\})$ are isomorphic; i.e.\ if and only if there are 
isomorphisms $\varphi_i:G_{F_i} \to G_{F_i}$ of trivial $G_{F_i}$-torsors for $i
\in I_v$ such that $(\varphi_r)_{F_k}\circ
  \mu_k = \tilde \mu_k\circ  (\varphi_{l})_{F_k}$ for each $(l,r,k) \in S_I$.
Let $g_i = \phi_i(1) \in G(F_i)$, where $1$ is the identity element of $G$.  The above equality on the $\varphi_i$   
is then equivalent to the equality
$g_rg_k=\tilde g_kg_l$ for $(l,r,k)\in S_I$; i.e., $g_k=g_r^{-1}\tilde g_kg_l$.  So the fibers of $\delta$ are as claimed, and the sequence is exact at $\prod_{k\in I_e}H^0(F_k,G)$.
\end{proof}

In this general setup, we let
$$\Sha_{\mc F}(G):=\operatorname{ker}\left(H^1(F,G) \to \prod_{i \in I_v}
H^1(F_i,G)\right)$$
be the kernel of the local-global map, i.e., the set of elements that map to
the trivial element.
Since $G$ need not be
commutative, the sequence in the theorem is exact just as a sequence of
pointed sets, not of groups; and so the kernel of a map does not determine the
fibers of the map.  Nevertheless, using standard techniques from nonabelian
cohomology, we can describe the fibers in terms of twists.

Recall (e.g.\ from Section~I.5.3 of
\cite{Serre:CG}) that if $G$ is a linear algebraic group over a field $F$ and
if $\tau \in Z^1(F,G)$ is a cocycle, then there is a naturally associated
twist $G^\tau$ of $G$ by $\tau$.  Moreover there is a bijection
between $H^1(F, G)$ and $H^1(F, G^\tau)$
such that the neutral element of $H^1(F, G^\tau)$
corresponds to the class $[\tau]$ in $H^1(F, G)$; and this is functorial in $F$.  (See
\cite[Proposition I.5.3.35]{Serre:CG}.)  Applying this discussion to
$F$ and the fields $F_i$, for $i \in I_v$, we obtain:

\begin{cor} \label{fibers_of_six-term_seq}
If $\tau \in Z^1(F,G)$, then the fiber of $H^1(F,G) \to \prod_{i \in I_v}
H^1(F_i,G)$ that contains the class $[\tau] \in H^1(F,G)$ is in natural
bijection with $\Sha_{\mc F}(G^\tau)$ as pointed sets.
\end{cor}

Of course if $\tau$ is the trivial cocycle, then this just says that the kernel is $\Sha_{\mc F}(G)$.

The six-term exact sequence also yields the following, which will be
used to study the connection between $\Sha_{\mc F}(G)$ and $\Sha_{\mc F}(G/G^0)$:
\begin{cor} \label{abstract_Sha_quotient}
Under the hypotheses of Theorem~\ref{general_6-term_sequence}, consider a
short exact sequence of linear algebraic groups
 $1 \to N\to G \to \bar G \to 1$ for which the map $G(L) \to \bar G(L)$ is
surjective for every field extension $L/F$.
Then the cohomology sequence associated to $1 \to N \to G \to \bar G \to 1$
induces a short exact sequence of pointed sets
\[1 \to \Sha_{\mc F}(N) \to \Sha_{\mc F}(G) \to \Sha_{\mc F}(\bar G) \to 1.\]
\end{cor}

\begin{proof}
Consider the commutative diagram below, whose
rows are exact by the cohomology sequence referred to in the assertion (\cite{Serre:CG}, I.5.5,
Proposition~38),
together with the surjectivity of $G \to \bar G$ on rational points.
The columns are exact by Theorem~\ref{general_6-term_sequence}.
\[
\xymatrix{
1 \ar[r]
& \prod_{k \in I_e} H^0(F_k,N) \ar[d]^{\delta'} \ar[r]^{\phi_0''}
 & \prod_{k \in I_e} H^0(F_k,G) \ar[d]^{\delta} \ar[r]^{\psi_0''}
 & \prod_{k \in I_e} H^0(F_k,\bar G) \ar[d]^{\bar \delta} \ar[r]
 & 1 \\
1 \ar[r]
& H^1(F,N) \ar[d]^{\iota'} \ar[r]^{\phi_1}
 & H^1(F,G) \ar[d]^{\iota} \ar[r]^{\psi_1}
 & H^1(F,\bar G) \ar[d]^{\bar \iota}
 & \\
1 \ar[r]
& \prod_{i \in I_v} H^1(F_i,N) \ar[r]^{\phi_1'}
 & \prod_{i \in I_v} H^1(F_i,G) \ar[r]^{\psi_1'}
 & \prod_{i \in I_v} H^1(F_i,\bar G)
 &  \\
 }
\]

The lower two rows show that the maps $\phi_1$ and $\psi_1$ restrict to
maps $\Sha_{\mc F}(N) \to \Sha_{\mc F}(G)$ and $\Sha_{\mc F}(G) \to
\Sha_{\mc F}(\bar G)$, with the former having trivial kernel and the
composition of these two restrictions being trivial.  The asserted exactness now follows from a diagram chase, using the definition of $\Sha_{\mc F}$.
\end{proof}

We will be interested in cases when the map $\Sha_{\mc F}(G)\to
\Sha_{\mc F}(\bar G)$ in Corollary~\ref{abstract_Sha_quotient} is a
bijection. The next proposition gives a criterion for this.
\begin{prop}\label{trans_factor}
Under the hypotheses of Corollary~\ref{abstract_Sha_quotient}, the following
are equivalent:
\begin{enumerate}
\renewcommand{\theenumi}{\roman{enumi}}
\renewcommand{\labelenumi}{(\roman{enumi})}
\item The map $\Sha_{\mc F}(G)\to \Sha_{\mc F}(\bar G)$ is a bijection of
  pointed sets.
\item\label{xx} For every pair of elements $(g_k), (\tilde g_k)\in \prod\limits_{k\in I_e}
  G(F_k)$ that have the same image in $\prod\limits_{k \in I_e}\bar G(F_k)$,
there exist elements $(g_i)\in \prod\limits_{i\in I_v}
  G(F_i)$ such that $g_k=g_r^{-1}\tilde g_k g_l$ whenever $(l,r,k)\in S_I$.
\end{enumerate}
\end{prop}
\begin{proof}
Assume that $\Sha_{\mc F}(G)\to \Sha_{\mc F}(\bar G)$ is a bijection,
and consider $(g_k), (\tilde g_k)\in \prod\limits_{k\in I_e} G(F_k)$ as in~(\ref{xx}). By
Theorem~\ref{general_6-term_sequence}, these define elements $\alpha,
\tilde\alpha\in \Sha_{\mc F}(G)$, which map to the same element
in $\Sha_{\mc F}(\bar G)$ since $(g_k), (\tilde g_k)$ have the same image in
$\prod\limits_{k \in I_e}\bar G(F_k)$. By assumption, this implies $\alpha=\tilde
\alpha$, which in turn implies that the images of $(g_k)$ and $(\tilde g_k)$
under the coboundary map are the same. Now use the last assertion of
Theorem~\ref{general_6-term_sequence}.

On the other hand, assume~(\ref{xx}). By Corollary~\ref{abstract_Sha_quotient}, it suffices to
show injectivity.  Consider two elements $\alpha, \tilde\alpha \in \Sha_{\mc
  F}(G)$ that have the same image in $\Sha_{\mc F}(\bar G)$.
Let $(g_k),(\tilde g_k) \in \prod_{\mc F}G(F_k)$ be preimages
under the coboundary map $\delta$ in Theorem~\ref{general_6-term_sequence}, and let
$(\bar g_k), (\bar{ \tilde g}_k)$ be their images in $\bar G(F_k)$. By
assumption, $(\bar g_k)$ and $(\bar{\tilde g}_k)$ induce the same
element in $\Sha_{\mc F}(\bar G)$; hence by the last assertion of
Theorem~\ref{general_6-term_sequence}, there exist $(\bar h_i) \in \bar
G(F_i)$ for $i\in I_v$ such that $\bar{ g}_k=\bar h_r^{-1} \bar{\tilde g}_k \bar
h_l$ whenever $(l,r,k)\in S_I$.
By the surjectivity hypothesis on $G\to \bar G$, there exist preimages $h_i\in
G(F_i)$ for the elements $\bar h_i$.  Replacing  $\tilde{g}_k$ by $h_r\tilde
g_kh_l$, we may assume that $(g_k)$ and $(\tilde
g_k)$ have the same image in $\prod_{k \in I_e} \bar G(F_k)$.
By the assumption of~(\ref{xx}) combined with the last assertion of
Theorem~\ref{general_6-term_sequence}, this shows that $(g_k)$ and $(\tilde
g_k)$ map to the same element under $\delta$, i.e., $\alpha=\tilde \alpha$,
thereby proving injectivity.
\end{proof}

Before we apply the results of this section in a more concrete setup, we state one more
corollary to Theorem~\ref{general_6-term_sequence}.

\begin{cor}\label{homogeneous}
Under the hypotheses of Theorem~\ref{general_6-term_sequence}, assume that
$\Sha_{\mc F}(G)=1$ and that $H$ is a $G$-variety over $F$ for which
$G(F_k)$ acts
transitively on $H(F_k)$ for all $k\in I_e$.
If $H(F_i)\neq \varnothing$ for each $i \in I_v$, then
$H(F)\neq \varnothing$.
\end{cor}
\begin{proof}
By the six-term exact sequence, the triviality of $\Sha_{\mc F}(G)$ is
equivalent to 
simultaneous factorization for $G$ over $\mc F$.
The proof is then the same as in \cite{HHK},
Theorem~3.7 (where an additional hypothesis on the group ensured the
simultaneous factorization property).
\end{proof}

In other words, a local-global principle for torsors implies a local-global
principle for homogeneous spaces (in our setup), hence the corollary allows us
to restrict our attention to torsors. 
We will return to other homogeneous spaces in Section~\ref{subsec_homogeneous}.

\section{Patching over Arithmetic Curves} \label{setup}

In the previous section, we saw that patching for vector spaces implies
patching for torsors, and that this yields a six-term exact sequence for Galois cohomology.
We wish to apply these results to linear algebraic groups defined over the
function field $F$ of an arithmetic curve. In order to do so, we use a setup
introduced in \cite{HH:FP}.

\begin{notation} \label{notn_fields}
Let $T$ be a complete discrete valuation ring with uniformizer $t$, fraction
field~$K$, and residue field $k$.  Let $F$ be a one-variable function field
over $K$, and let $\wh X$ be a \textit{normal model} of $F$, i.e.\ a normal
connected projective $T$-curve with function field $F$. For each point $P$ of
the closed fiber $X \subset \wh X$,  let $\wh R_P$ be the completion of the local ring $R_P$ of
$\wh X$ at $P$ (with respect to its maximal ideal), and $F_P$ be the fraction
field of $\wh R_P$.  For each subset $U$ of $X$ that is contained in an
irreducible component of $X$ and does not meet other components, let $\wh R_U$
be the $t$-adic completion of the subring of $F$ consisting of rational
functions that are regular on~$U$, and let $F_U$ be its fraction field.
For each branch of $X$ at a closed point $P$, i.e., for each height one
prime $\wp$ of $\wh R_P$ that contains $t$, let $\wh R_\wp$ be the completion
of the local ring of $\wh R_P$ at $\wp$, and let $F_\wp$ be its fraction
field.
\end{notation}

In the setup of Notation~\ref{notn_fields}, in \cite{HH:FP} and \cite{HHK}, we
considered an inverse system of fields, arising from a choice as follows:

\begin{notation} \label{notn_patches}
In Notation~\ref{notn_fields}, let $\mc P$ be a non-empty finite set of closed points of
$X$ that contains all the closed points at which distinct irreducible
components of $X$ meet.  Let $\mc U$ be the set of connected components of the
complement of $\mc P$ in $X$.  Let $\mc B$ be the set of branches of $X$ at points of $\mc P$.
\end{notation}

This yields a finite inverse system of fields
$F_P, F_U, F_\wp$ of $F$ (for $P \in \mc P$, $U \in \mc U$, $\wp
\in \mc B$), where $F_P,F_U \subset F_\wp$ if $\wp$ is a branch of $X$ at
$P$ lying on the closure of $U$. The system is indeed a factorization inverse system.

In \cite{HH:FP}, Proposition~6.3, it was shown that the inverse limit of this system is $F$, provided
that $\mc P = f^{-1}(\infty)$ for some finite morphism $f:\wh X \to \mbb
P^1_T$, where $\infty$ is the point at infinity on the closed fiber of $\mbb
P^1_T$.  But in fact, this last assumption is satisfied automatically, as the
following strengthening of \cite[Proposition~6.6]{HH:FP} shows:

\begin{prop} \label{map_to_line}
In the situation of Notation~\ref{notn_fields},
let $S$ be a finite set of closed points of $\wh X$ and write $\infty$ for the
point at infinity on $\mbb P^1_k \subset \mbb P^1_T$. Then there is a finite
$T$-morphism $f:\wh X \to \mbb P^1_T$ such that $S = f^{-1}(\infty)$ if and
only if $S$ meets each irreducible component of $X$ non-trivially.
In particular, there is such an $f$ for $S = \mc P$ as in Notation~\ref{notn_patches}.
\end{prop}

\begin{proof}
By \cite{Liu}, Corollary~5.3.8, a
divisor on a proper scheme is ample if and only if its restriction to each
irreducible component is ample.  Since an effective divisor on an irreducible
projective curve over a field is ample if and only if it is non-zero, it
follows that an effective divisor on $X$ is ample if and only if it meets each
irreducible component of $X$ non-trivially.

Since $\infty$ defines an ample divisor on $\mbb P^1_k$, the forward
implication of the proposition now follows from the fact that the inverse
image of an ample divisor under a finite surjective morphism is ample
(\cite{Liu}, Remark~5.3.9).

We now prove the reverse implication; so assume that $S$ meets each
irreducible component of $X$ non-trivially.  For each point $P \in S$, the
maximal ideal $\mathfrak m_P$ of the local ring $R_P \subset F$ strictly
contains the ideal $I_P$ that defines the reduced closed fiber of $\wh X$ in
$\Spec(R_P)$; so we may choose an element $r_P \in \mathfrak m_P$
that does not lie in $I_P$.  The element $r_P$ defines a non-trivial effective Cartier
divisor $(r_P)$ on $\Spec(R_P)$.  Now since $T$ is complete, the connected components of a closed subset of the projective $T$-scheme $\wh X$ are in bijection with the connected components of its intersection with the closed fiber $X$.  
Thus $(r_P)$ is the restriction to $\Spec(R_P)$ of an 
effective Cartier divisor $\wh D_P$ on $\wh X$ whose support meets $X$ precisely at $P$.
Let $\wh D = \sum_{P\in S} \wh D_P$. The restriction of $\wh D$ to $X$ has
support $S$, and so meets each irreducible component of $X$ non-trivially.
Thus this restriction is an ample divisor on $X$.  Hence by \cite{Liu},
Corollary~5.3.24 (or by \cite{EGA3}, Th\'eor\`eme~4.7.1), $\wh D$ is an ample
divisor on $\wh X$.  Replacing $\wh D$ by a multiple (which corresponds to
replacing each $r_P$ by a power), we may assume that $\wh D$ is very ample (\cite{EGA2}, Proposition~4.5.10(ii)(e$'$)).
Hence for some $n$ there is an embedding $i:\wh X \to \mbb P^n_T$ such that
$\wh D = i^*(H_1)$ for some $T$-hyperplane $H_1 \subset \mbb P^n_T$, given by
a linear form $F_1$ over $T$.  Choose a basic open subset of
$\mbb P^n_k$ that contains the finite set $i(S)$. Its complement in $\mbb P^n_k$ is a $k$-hypersurface, hence is given by a homogeneous $k$-form of some degree $d$.  Lifting the
coefficients of this form from $k$ to $T$, we obtain a $T$-hypersurface $H_2
\subset \mbb P^n_T$, defined by a homogeneous $T$-form $F_2$ of degree $d$,
such that the support of $i^*(H_2)$ does not meet $S$ and hence is disjoint
from the support of $\wh D$.  Consider the rational function $f :=
i^*(F_2/F_1^d)$ on $\wh X$, whose zero divisor is $i^*(H_2)$ and whose pole
divisor is $i^*(dH_1) = d\wh D$.  Since those divisors have disjoint supports,
$f$ defines a $T$-morphism $\wh X \to \mbb P^1_T$, such that the inverse image
of $\infty \in \mbb P^1_k$ is $S$.  It remains to show that $f$ is finite.

The pole locus of $f$ on $\wh X$, viz.\ the support of $\wh D$, meets each
irreducible component of $X$ at a point of $S$.  The zero locus of $f$ on $\wh
X$ is the support of the very ample divisor $i^*(H_2)$, which therefore also
meets each irreducible component of $X$ non-trivially.  Hence $f$ is
non-constant on each irreducible component of $X$, having both a zero and a
pole there.  Similarly, $f$ is non-constant on the general fiber of $\wh X$,
which is irreducible since $\wh X$ itself is.  Thus $f$ does not contract any
irreducible component of a fiber of $\wh X \to T$, and so is a quasi-finite
(i.e.\ finite-to-one) $T$-morphism.  But since the proper morphism $\wh X \to
T$ factors through $f:\wh X \to \mbb P^1_T$, it follows that $f$ is proper
(\cite{Hts}, Corollary~II.4.8(e)).  Being quasi-finite and proper, $f$ is
finite by \cite{EGA4.3}, Th\'eor\`eme~8.11.1.

The last part of the proposition is immediate from the first part if $X$ is
irreducible, since $\mc P$ is non-empty.  On the other hand, if $X$ is
reducible, then each component must
meet some other component by connectivity of $X$, and thus each component contains a
point of $\mc P$.  So again the hypothesis of the first part is satisfied.
\end{proof}

{}From the observation prior to Proposition~\ref{map_to_line} together with
\cite{HH:FP}, Theorem~6.4, we obtain:

\begin{cor}\label{patching_equivalence}
The inverse system given by Notation~\ref{notn_patches} is a factorization
inverse system with inverse
limit $F$. For this system, the base change functor
$\beta: \Vect(F) \to \PP(\mc F)$ is an equivalence of categories.
\end{cor}

Moreover, for each $\wp\in \mc B$, exactly one of the two indices
dominating it is from $\mc P$ and one from $\mc U$. In particular, there is a natural
and uniform way to order those indices: We consider triples $(P,U,\wp)$ such
that $\wp$ is a branch at $P$ lying on the closure of $U$.

As a consequence of Corollary~\ref{patching_equivalence}, the results of
Section~\ref{patching} can be applied to the inverse system given in
Notation~\ref{notn_patches}.  In particular, we obtain
the following Mayer-Vietoris type sequence from Theorem~\ref{general_6-term_sequence}:

\begin{thm}\label{6-term_sequence}
Under Notation~\ref{notn_patches}, and for any linear algebraic group $G$ over
$F$, we have an exact sequence of pointed sets
\[\xymatrix{
1 \ar[r] & H^0(F,G) \ar[r] &
\prod_{P \in \mc P} H^0(F_P,G) \times \prod_{U \in \mc U} H^0(F_U,G) \ar[r] &
\prod_{\wp \in \mc B} H^0(F_\wp,G) \ar@<-2pt> `d[l]
`[lld] [lld] \\
& H^1(F, G) \ar[r] & \prod_{P \in \mc P} H^1(F_P,G) \times \prod_{U \in \mc U}
H^1(F_U,G) \ar@<.5ex>[r] \ar@<-.5ex>[r] & \prod_{\wp \in \mc B}
H^1(F_\wp,G). }\]
Moreover, two elements $(g_{\wp}), (\tilde g_{\wp})\in \prod_{\wp \in \mc B} H^0(F_{\wp},G)$
have the same image under the coboundary map if and only if there exist
$g_\xi \in G(F_\xi)$ for each $\xi \in \mc P \cup \mc U$ such that $g_{\wp}=g_U^{-1}\tilde g_{\wp} g_P$ whenever $\wp$ is a branch at $P$ lying on the closure of $U$.
\end{thm}

In this setup, the kernel of the local-global map will be denoted by
\[\Sha_{\wh X,\mc P}(F,G) := \ker\bigl(H^1(F,G) \to \prod_{\xi \in \mc P \cup \mc U}
H^1(F_\xi,G) \bigr),\]
or for short just by $\Sha_{\mc P}(F,G)$ if $\wh X$ is understood.
This kernel equals $\Sha_{\mc F}(G)$, for the inverse system of fields $\mc F$ associated to $\mc P, \mc U, \mc B$.
Note that $\mc P \subset X$ determines $\mc U$, and so the notation need not
make explicit mention of $\mc U$.

\begin{cor} \label{Sha_double-coset}
The coboundary of the above exact sequence induces a bijection
\[\prod_{U \in \mc U} G(F_U) \ \big\backslash \ \prod_{\wp \in \mc B}
G(F_\wp) \ \big/ \ \prod_{P\in \mc P} G(F_P) \to \Sha_{\mc P}(F,G) \]
of pointed sets.
\end{cor}

\begin{remark}
The above double coset space is reminiscent of the classical adelic double coset space
$$\Sigma_{G,K}:=G(K)\ \big \backslash \ G({\mathbb A}_K)\ \big/ \ \prod \limits_{v\in\Omega_K} G(\mc O_v),$$
where $K$ is a function field over a finite field with set of discrete valuations $\Omega_K$, complete local rings $\mc O_v$, and adeles ${\mathbb A}_K$. This space is indeed also used in the study of $G$-torsors over $K$ and of $\Sha(K,G)$. See, for example, \cite{Conrad}, Sections~1.2 and~1.3.
\end{remark}

\section{The case of rational groups}\label{sec_rational}

In \cite{HHK}, certain local-global principles were shown to hold for
connected rational linear algebraic groups. In this manuscript, we permit
consideration of groups that are not necessarily connected. For an arbitrary
linear algebraic group $G$ over an infinite field $k$, we say that $G$ is
\textit{rational} over $k$ (or {\em $k$-rational}) if every connected component of $G$ is a rational variety
over $k$.  It can be shown that over an algebraically closed field, every linear
algebraic group is rational. Rationality is equivalent to the condition that the identity component
$G^0$ is $k$-rational and every component has a $k$-point. 
It is also equivalent to the condition that $G^0$ is $k$-rational, that 
$G/G^0$ is a constant finite $F$-group scheme, and that 
the map $G(F) \to (G/G^0)(F)$ is surjective.
In particular, for the inverse system of fields considered in Section~\ref{setup}, the hypothesis of
Corollary~\ref{abstract_Sha_quotient} is satisfied for any $F$-rational linear algebraic
group $G$, with $N=G^0$ (also using Corollary~\ref{patching_equivalence}). 
Hence for a rational linear algebraic group $G$ over $F$, and choice of $\wh X$ and $\mc P$, we obtain a
short exact sequence of pointed sets
\[1 \to \Sha_{\wh X,\mc P}(F,G^0) \to \Sha_{\wh X,\mc P}(F,G) \to \Sha_{\wh X,\mc P}(F,G/G^0) \to 1\]
induced from the corresponding cohomology sequence.

Moreover, Proposition~\ref{map_to_line} shows that $\mc P$ as in
Notation~\ref{notn_patches} satisfies the hypotheses made
in Section~3 of \cite{HHK}. Theorem~3.7 of \cite{HHK} then implies that
$\Sha_{\mc P}(F,G^0)=1$. Summing up the discussion, we find that
the natural surjection $\Sha_{\mc P}(F,G) \to \Sha_{\mc P}(F,G/G^0)$
has trivial kernel for rational linear algebraic groups $G$. 

Note that whereas $\Sha_{\mc P}(F,G) \to \Sha_{\mc P}(F,G/G^0)$
is in general just a map of pointed sets, it is a group homomorphism in the
case that $G$ is commutative, being the restriction to
$\Sha_{\mc P}(F,G)$ of the group homomorphism
$H^1(F,G) \to H^1(F,G/G^0)$.  Hence if the linear algebraic group $G$ is
rational and commutative, the above says that this map is an isomorphism of groups.
But in general, a map of pointed sets can have trivial kernel without being
injective.

Nevertheless, we show below that (if $G$ is rational) the map $\Sha_{\mc
  P}(F,G) \to \Sha_{\mc P}(F,G/G^0)$ is indeed a bijection of pointed
sets. This reduces the study of $\Sha_{\mc P}(F,G)$ for rational linear algebraic groups
to the case of finite constant groups (i.e.\ to the study
of finite Galois extensions of the function field $F$).

We first prove a variant on Theorem~2.5 of~\cite{HHK}. We recall the situation considered there:

Let $\wh R_0$ be a complete discrete valuation ring  with uniformizer $t$ and
field of fractions $F_0$, and suppose $\wh R_0$ contains a subring $T
\subseteq \wh R_0$ that is also a complete discrete valuation ring with
uniformizer $t$.  We let $|\cdot|$ be a norm defined on $F_0$ defined by the $t$-adic
valuation. This induces a maximum norm on $\mbb A^n_{F_0}$ in the usual way.
Let $F_1, F_2 $ be subfields of $F_0$ containing~$T$.

\begin{prop} \label{small_factorization}
In the above situation, assume there are
$t$-adically complete $T$-submodules $V \subset F_1 \cap \wh R_0$, $W \subset F_2\cap \wh R_0$
satisfying $V+ W  =\wh R_0$. Assume moreover that $F_1$ is $t$-adically dense
in $F_0$.
Let $f = (f_1, \ldots, f_n): \mbb A^n_{F_0} \times \mbb A^n_{F_0} \dra \mbb
A^n_{F_0}$ be an $F_0$-rational map that is defined on a Zariski open set $U \subseteq \mbb A^n_{F_0}
\times \mbb A^n_{F_0}$ containing the origin $(0,0)$.  Suppose that for each $i$ we may write (in multi-index notation)
\[f_i = T_{1,i}(x) + T_{2,i}(y) + \sum_{|(\nu, \rho)| \geq 2} c_{\nu, \rho, i} x^\nu
y^\rho, \eqno{(*)}\]
for some elements $c_{\nu, \rho, i}$ in $F_0$, where $x=(x_1,\dots,x_n)$ and
$y=(y_1,\dots,y_n)$, and where $T_j = (T_{j,1},\dots,T_{j,n})$ is an
automorphism of the vector space $F_0^n$ for $j=1,2$.
Then there is a real number $\varepsilon > 0$
such that for all $a \in \mbb A^n(F_0)$ with $|a| \leq \varepsilon$, there
exist $v \in \m F_1^n$ and $w \in \m F_2^n$ such that $(v,w) \in U(F_0)$ and
$f(v, w) = a$.
\end{prop}

\begin{proof}
\textit{Case 1:} $T_1 = T_2$ is the identity transformation.

In this situation, the assertion was shown at~\cite[Theorem~2.5]{HHK}.  The
statement of that result had assumed that $f(u, 0) = u = f(0, u)$ whenever
$(u, 0)$ (resp.\ $(0,u)$) is in $U$, rather than assuming $(*)$.  But that
hypothesis was used only in order to obtain $(*)$ and the equality $f(0,0)=0$,
which itself follows from $(*)$.  This was done in the first paragraph of the
proof of~\cite[Theorem~2.5]{HHK}.  The remainder of the proof
of~\cite[Theorem~2.5]{HHK} used only $(*)$, and so in our situation that
argument carries over and the desired conclusion follows.
\smallskip

\textit{Case 2:} $T_1$ is the identity and $T_2$ is arbitrary.

The group $G := \GL_{n,F}$ is rational and connected over $F := F_1 \cap
F_2$ and $F_1$ is dense in $F_0$. 
Hence Theorem~3.2 of~\cite{HHK} applies, and asserts that for any $g \in
\GL_n(F_0)$ there exist $g_1 \in \GL_n(F_1)$ and $g_2 \in \GL_n(F_2)$ such
that $g=g_1g_2$.  Taking $g$ to be the matrix associated to $T_2$, we obtain
invertible linear transformations $S_i$ of $F_i^n$ for $i=1,2$ such that
$T_2 =S_1S_2^{-1}$.

Consider the $F_0$-rational map $\tilde f = S_1^{-1} \circ f \circ (S_1\times S_2) : \mbb
A^n_{F_0} \times \mbb A^n_{F_0} \dra \mbb A^n_{F_0}$.  This is defined on the
Zariski open set $\tilde U := (S_1\times S_2)^{-1}(U)$, which contains the origin $(0,0)$.
Write $\tilde f = (\tilde f_1,\dots,\tilde f_n)$. Since $\tilde f(0,0) = 0$,
we may write $\tilde f_i \in F_0[[x_1,\dots,x_n,y_1,\dots,y_n]]$ as
$\tilde f_i = L_{1,i}(x) + L_{2,i}(y) + \sum_{\nu,\rho} \tilde c_{\nu,
  \rho, i} x^\nu y^\rho$ for some $\tilde c_{\nu, \rho, i}$ in
$F_0$, where $L_{1,i}$ and $L_{2,i}$ are linear forms in $x_1,\dots,x_n$ and
$y_1,\dots,y_n$ respectively, where the sum ranges over $|(\nu, \rho)|
\geq 2$ and where $1 \le i \le n$.

Let $L_j:=(L_{j,1},\ldots,L_{j,n})$ for $j=1,2$.
Equating linear terms in the definition of $\tilde f$ then yields the equality
$S_1^{-1} \circ (\id + T_2) \circ (S_1\times S_2) = L_1 + L_2$, where $\id$ is the
identity map on $x_1,\dots,x_n$ and $T_2$ is viewed as a map on
$y_1,\dots,y_n$. That is, $L_1(x) + L_2(y) = S_1^{-1}(S_1(x) + T_2S_2(y)) =
x+y$.
We then have
\[
\tilde f_i = x_i + y_i + \sum_{|(\nu, \rho)| \geq 2} \tilde c_{\nu, \rho, i}
x^\nu y^\rho
\]
for each $i$.  Hence Case~1 applies, and there is a real number $\tilde
\varepsilon> 0$ as in the assertion of the theorem. Since the linear map $S_1$ is continuous,
there is an $\varepsilon > 0$ such that if $|a| \le \varepsilon$ then
$|S_1^{-1}(a)| \leq \tilde \varepsilon$. Now suppose that $a \in \mbb
A^n(F_0)$ with $|a| \leq \varepsilon$, and let $\tilde a := S_1^{-1}(a)$.
By the assertion of the theorem in Case~1, there exist $\tilde v \in F_1^n$ and $\tilde w
\in F_2^n$ such that $(\tilde v,\tilde w) \in \tilde U(F_0)$ and $\tilde
f(\tilde v, \tilde w) = \tilde a$.

Let $v = S_1(\tilde v) \in F_1^n$ and $w = S_2(\tilde w) \in F_2^n$. Thus
$(v,w) = (S_1\times S_2)(\tilde v,\tilde w) \in U$.  The equality $f =
S_1\circ \tilde f \circ (S_1\times S_2)^{-1}$ then yields $f(v,w) = S_1(\tilde
f(\tilde v,\tilde w)) = S_1(\tilde a) = a$.  Thus $\varepsilon$ has the
required properties.  This concludes the proof in Case~2.

\smallskip

\textit{Case 3:} General case.

Let $\tilde f = T_1^{-1} \circ f : \mbb A^n_{F_0} \times \mbb A^n_{F_0} \dra
\mbb A^n_{F_0}$.  Thus $\tilde f$ is defined on $U$.  Write $\tilde f =
(\tilde f_1,\dots,\tilde f_n)$ and  $S = (S_1,\dots,S_n) := T_1^{-1}T_2$.
Then for each $i$ we may write $\tilde f_i = x_i + S_i(y) + \sum_{|(\nu,
  \rho)| \geq 2} \tilde c_{\nu, \rho, i} x^\nu y^\rho$ for some $\tilde
c_{\nu, \rho, i}$ in $F_0$.  So by Case~2, there is a real number $\tilde
\varepsilon > 0$ as in the statement of the theorem.
Since $T_1^{-1}$ is continuous, there is an $\varepsilon > 0$ such
that if $|a| \le \varepsilon$ then $|T_1^{-1}(a)| \leq \tilde \varepsilon$.
Thus if $a \in \mbb A^n(F_0)$ satisfies $|a| \leq \varepsilon$ then by the
assertion of the theorem (in Case~2), there
exist $v \in F_1^n$ and $w \in F_2^n$ such that $(v,w) \in U(F_0)$ and $\tilde
f(v, w) = T_1^{-1}(a)$, and hence $f(v,w)=a$.
\end{proof}

We now return to the situation of Notation~\ref{notn_patches}.
Let $G$ be a rational linear algebraic group over $F$, and let $G^0$ be its
identity component.

\begin{thm} \label{sha_inj}
Under Notation~\ref{notn_patches}, if $G$ is a rational linear algebraic group
over $F$, then the natural map $\Sha_{\mc P}(F,G) \to \Sha_{\mc P}(F,G/G^0)$ is a bijection.
\end{thm}

\begin{proof}
Write $\bar G:=G/G^0$; this is a finite constant group by the rationality assumption.
By Corollary~\ref{patching_equivalence} and the rationality of $G$, the hypotheses of Proposition~\ref{trans_factor} are satisfied.  Hence it suffices to show that for any two elements
$(g_{\wp}),(\tilde g_{\wp}) \in \prod_{\mc B}G(F_{\wp})$  that map to the same
element in $\prod_{\mc B}\bar G$, there exist elements
$g_\xi\in G(F_\xi)$ for all $\xi\in {\mc P} \cup {\mc U}$ satisfying the condition that $g_{\wp}=g_U^{-1}\tilde g_{\wp} g_P$ whenever $\wp$ is a branch at $P$ lying on the
closure of $U$.  This is equivalent to the condition that 
$h_\wp = g_\wp^{-1} g_U g_\wp g_P^{-1}$ for all such triples $P,U,\wp$, where 
$h_{\wp}:=g_{\wp}^{-1}\tilde g_{\wp} \in G^0 (F_{\wp})$.  We now proceed to show that there exist such elements $g_\xi$.

\smallskip

\textit{Case 1:} $\wh X = \mbb P^1_T$ and $\mc P =
\{\infty\}$.

In this case, $\mc U$ contains a single element $U = \mbb A^1_k$, and $\mc B$ consists just of the unique branch $\wp$ at $\infty$.  For short write  $g=g_\wp$ and $h=h_\wp$.  Also write $F_1=F_P$, $F_2=F_U$, and $F_0=F_\wp$.
Define the map $\gamma:G^0(F_0)\times G^0(F_0)\to
G^0(F_0)$ by $(g_1,g_2)\mapsto g^{-1}g_2gg_1^{-1}$.
We wish to show that $h$ is the
image under $\gamma$ of a pair $(g_1,g_2)\in G^0(F_1)\times
G^0(F_2)$. 

Since $G^0$ is rational, there is an
$F$-isomorphism $\phi:U' \to U$ from a connected open neighborhood of the
identity in $G$ to an open neighborhood of the origin in $\mbb A^n_{F'}$ (for a
suitable $n$).
The map $\phi$ carries $\gamma$ to an
$F_0$-rational map
$f : \mbb A^n_{F_0} \times \mbb A^n_{F_0} \dra \mbb A^n_{F_0}$
that is defined on a Zariski open set $U \subseteq \mbb A^n_{F_0}
\times \mbb A^n_{F_0}$ containing the origin $(0,0)$.  That is, $f \circ (\phi
\times \phi)$ agrees with $\phi \circ \gamma$ as a rational map.
Moreover $f(0,0)=0$.  So as an element of
$F_0[[x_1,\dots,x_n,y_1,\dots,y_n]]$, $f=(f_1,\dots,f_n)$ has the form $(*)$
of Proposition~\ref{small_factorization}, with $T_i$ being a linear
transformation of the vector space $F_0^n$, corresponding to a matrix $A_i \in
\Mat_n(F_0)$.  Here $(A_1,A_2)$ is the Jacobian matrix of $f$ at the origin,
and $A_1$ (resp.\ $A_2$) is the Jacobian matrix of the restriction of $f$ to
$y_1=\dots=y_n=0$ (resp.\ to $x_1=\dots=x_n=0$).  The matrices $A_1, A_2$ are
invertible, since the restrictions $\gamma(\cdot,1)$ and $\gamma(1,\cdot)$ are
invertible; and hence $T_i$ is an automorphism of $F_0^n$ for $i=1,2$.
Let $V = \wh R_1$ and $W = \wh R_2$.
Since every element of $\wh R_0/(t) = k((x^{-1}))$ is the sum of elements of
$\wh R_1/(t) = k[[x^{-1}]]$ and $\wh R_2/(t) = k[x]$, a standard induction
argument shows that $V+W=\wh R_0$.
So the hypotheses of Proposition~\ref{small_factorization} are satisfied.
By the conclusion of that proposition, there is an $\varepsilon > 0$ such that
if $a \in \mbb A^n_{F_0}$ satisfies $|a| \le \varepsilon$ then $a = f(v,w)$
for some $v \in F_1^n$ and $w \in F_2^n$.

Now $F_1$ is dense in $F_0$ (see the proof of \cite{HHK}, Theorem~3.4). Hence
we may apply Lemma~3.1 of~\cite{HHK} to the rational group $G^0$ and the field
$F_0$ to obtain an element $h_1 \in G^0(F_1)$ such that $|\phi(hh_1)| \le
\varepsilon$ (for details see the last paragraph of the proof of \cite{HHK}, Theorem~3.2).
Thus $\phi(hh_1) = f(v,w)$ for some $v \in F_1^n$ and some $w \in
F_2^n$.
Applying $\phi^{-1}$, we obtain the conclusion that $hh_1$ is the
image under $\gamma$ of a pair $(g_1',g_2) \in G^0(F_1) \times
G^0(F_2)$, i.e., $hh_1=g^{-1}g_2g(g_1')^{-1}$. But then
$h=g^{-1}g_2g(g_1')^{-1}h_1^{-1}$, or equivalently, $h=\gamma(g_1,g_2)$ for
$g_1:=h_1g_1'$. This finishes the proof in Case~1.

\smallskip

\textit{Case 2:} General case.

By Proposition~\ref{map_to_line},
there is a finite morphism $\wh X \to \mbb P^1_T$ such that $\mc P =
f^{-1}(\infty)$.
Write $F'$ for the function field $K(x)$ of $\mbb P^1_T$, and let
$d=[F:F']$. As in Notation~\ref{notn_fields} for $\mbb P^1_T$, we may consider
rings and fields associated to $\infty$, $\mbb A^1_k$, and to the unique branch on $\mbb
P^1_k$ at infinity; we call these $\wh R_1, F_1,\wh R_2, F_2, \wh R_0, F_0$, respectively.
Let $G'$ be the Weil restriction $R_{F/F'}(G)$ of $G$ (which exists even if $F/F'$ is inseparable; see \cite{BLR:Neron}, Section~7.6). By the defining
property of the Weil restriction, there is a natural isomorphism
$G'(F_0)=G(F_0\otimes_{F'} F)$, and the latter equals $\prod_{\mc B}
G(F_{\wp})$ by \cite{HH:FP}, Lemma~6.2(a). Similarly, $G'(F_1)=\prod_{\mc P}
G(F_P)$ and $G'(F_2)=\prod_{\mc U} G(F_U)$. Under these identifications, write
$(g_{\wp})=:g\in G'(F_0)$, $(\tilde g_{\wp})=:\tilde g\in G'(F_0)$, and
$(h_{\wp})=:h\in G'^0(F_0)$. 
Thus $h=g^{-1}\tilde g \in G'^0(F_0)$.
Since $G^0$ is (geometrically) connected and smooth, so is 
$R_{F/F'}(G^0)$ (\cite{Oesterle}, Proposition~A.3.7; 
\cite{CoGaPr}, Proposition~A.5.9).  Hence $R_{F/F'}(G^0)$ is
the identity component of $G'$, and moreover it is rational.

By Case~1, there exist $g_1 \in G'^0(F_1)$ and $g_2 \in G'^0(F_2)$ such that 
$h= g^{-1}g_2gg_1^{-1}$.  Viewing this as an equation in $\prod_{\mc B}G(F_{\wp})$
then yields the desired elements $g_\xi$.
\end{proof}

\begin{remark}
Theorem~\ref{sha_inj} and Corollary~\ref{homogeneous} above can be regarded as analogous to  Cor.~3 and Cor.~1 of \cite[Th\'eor\`eme~III.2.4.3]{Serre:CG}.  The proofs of our results here, however, are quite different from the proofs there.
\end{remark}

\section{Split covers}\label{sec_splitcovers}

In this section we consider a different and more canonical local-global map, with respect to {\em all} points of the closed fiber (including the generic points of
components) of a curve
$\wh X$ over a complete discrete valuation ring $T$, as in
Notation~\ref{notn_fields}.

We define
\[\Sha_{\wh X,X}(F,G) := \ker\bigl(H^1(F,G)\rightarrow \prod\limits_{P\in X} H^1(F_P, G)\bigr)\]
for a linear algebraic group $G$ over our field $F$, 
and for short we write $\Sha_X(F,G)$ if the model $\wh X$ is understood.
For finite subsets
$\mc P \subseteq \mc P'$ of the closed fiber $X$ as in Notation~\ref{notn_patches}, we have the containments
\[\Sha_{\mc P}(F,G) \subseteq \Sha_{\mc P'}(F,G) \subseteq \Sha_X(F,G) \subseteq H^1(F,G).\]
Namely, if $P \in U$ then $F_U \subset F_P$, hence any torsor with an
$F_U$-point has an $F_P$-point; and if $U' \subseteq U$ then similarly a
torsor with an $F_U$-point has an $F_{U'}$-point.

Below we show that for $G$ rational, $\Sha_X(F,G)$
equals $\Sha_{\mc P}(F,G)$ (for any choice of ${\mc P}$), and we give its
structure in terms of a certain quotient $\pi_1^\s(\wh X)$ of the \'etale
fundamental group of $\wh X$ (Theorem~\ref{sha_fsplit}).

The study of $\Sha_X(F,G)$ relies on the study of split covers.
We will say that a degree $n$ morphism $h:\wh Y \to
\wh X$ of normal projective $T$-curves is \textit{split} over a
point $P$ of $\wh X$ if $\wh Y \times_{\wh X} P$ consists
of $n$ copies of $P$.
In this case $h$ is \'etale over $P$ since the fiber $\wh Y \times_{\wh X} P$
is reduced and separable (in fact trivial) over $P$.
More generally, given a morphism $S \to \wh X$, we will similarly say that $h$
is \textit{split} over $S$  (or over $A$, in case $S = \Spec(A)$) if $\wh Y
\times_{\wh X} S$ is a disjoint union of copies of $S$. We say that $h$ is a
\textit{split cover} if it is split over every point $P$ of $\wh X$ except
possibly the generic point.  Thus every split cover of $\wh X$ is \'etale.

The inverse limit of the Galois groups of pointed Galois (connected) split covers of $\wh X$ will be denoted by
$\pi_1^{\s}(\wh X)$; this is a quotient of  the \'etale fundamental group
$\pi_1(\wh X)$.  (Here the base points on the covers lie over a chosen base point on $\wh X$; and up to isomorphism, $\pi_1^{\s}(\wh X)$ is independent of this choice.)
For any
finite constant group $G$, 
$\Hom(\pi_1^\s(\wh X),G)$ classifies
the pointed $G$-Galois split covers of $\wh X$, with $\Epi(\pi_1^\s(\wh X),G)$ classifying those that are connected.  
Similarly, the (unpointed)  $G$-Galois split
covers of $\wh X$ are classified by $\Hom(\pi_1^\s(\wh X),G){/\!\sim}$, where
the equivalence relation is given by conjugation by $G$. Here $\Hom(\pi_1^\s(\wh X),G){/\!\sim}$ is viewed as a subset of $\Hom(\Gal(F^\sep/F),G){/\!\sim} = H^1(F,G)$, which classifies the $G$-Galois \'etale algebras over $F$. If $G$ is abelian, then conjugation is trivial. Moreover, $H^1(F,G)$ is a group in that case, and $\Hom(\pi_1^\s(\wh X),G)$ is a subgroup.

Observe that if $h:\wh Y \to \wh X$ is a finite morphism of $T$-curves and $P$
is a point of the closed fiber $X \subset \wh X$, then $h$ is split over $P$ if and only if $h$ is split
over the ring $\wh R_P$, by Hensel's Lemma (which applies since being split
implies being \'etale).
If moreover $\wh Y$ is normal then these conditions are also equivalent to $h$
being split over the field $F_P$, since in that case the fiber over $\wh R_P$
is the normalization of $\wh R_P$ in the fiber over $F_P$. For $U \in \mc U$
as in Notation~\ref{notn_patches},
the corresponding equivalences hold for $U$, $\wh R_U$, $F_U$.  Instead of Hensel's Lemma,
this uses \cite[Lemma~4.5]{HHK}, which applies since the morphism $h$ is \'etale and hence smooth.

\begin{prop} \label{split_points}
Suppose that $h:\wh Y \to \wh X$ is a finite morphism of normal
projective $T$-curves, with $\wh X$ connected. 
Then $h$ is a split cover if and only if it is split
over $F_P$ for every (not necessarily closed) point $P$ of the closed fiber $X$.
\end{prop}

\begin{proof}
The forward direction follows from the above observation, which also shows
that if $h$ is split over every $F_P$ then it is it is split over every point
$P$ of $X$.

To complete the proof of the converse, it remains to show that $h$ is split at
each codimension one point of $\wh X$ that is not supported on $X$.  Such a
point is a closed point of the general fiber of $\wh X$, with residue field a
finite extension $K'$ of $K$.  Its closure $Z$ in $\wh X$ is of the form
$\Spec(T')$ for some finite extension $T'$ of $T$,
since $\wh X$ is proper over $T$.  So $T'$ is
a $t$-adically complete one-dimensional local domain (but
not necessarily normal).  The maximal ideal of $T'$ corresponds to the unique closed
point $P$ of $Z$, which lies on $X$; and since $T'$ is complete, it follows that the inclusion $Z \hookrightarrow \wh X$ factors through $\Spec(\wh R_P)$.  
Since $h$ is split over $F_P$, it is also split over $\wh R_P$, hence also over $Z$, and thus over the generic point of $Z$, which is the given 
codimension one point of $\wh X$.
\end{proof}

\begin{remark}\label{split_remark}
\renewcommand{\theenumi}{\alph{enumi}}
\begin{enumerate}
\item \label{bijection}
Proposition~\ref{split_points} shows that there is a natural bijection between the set of split
covers of $\wh X$ and the set of finite morphisms to the closed fiber $X$ that
are split over every point.  This follows from  Hensel's Lemma and the fact
that every \'etale cover of $X$  lifts uniquely to an \'etale cover of $\wh X$
(\cite{SGA}, Theorem~X.2.1).

\item In Proposition~\ref{split_points}, if
the residue field $k$ of $T$ is finite, then
Chebotarev's Density Theorem applies.  Hence
$h:\wh Y \to \wh X$ is a split
cover if and only if it is split over every \textit{closed} point of $\wh X$
(i.e.\ of $X$).  A cover satisfying this latter splitness condition is referred
to in \cite{Saito}, Definition~II.2.1, as a \textit{c.s.\ cover}.  If the
residue field $k$ is not assumed to be finite (e.g.\ if it is algebraically
closed), then that condition is in general weaker than being a split cover.
Compare Corollary~\ref{graph_pi1} below to Proposition~2.2 and Theorems~2.4 and~7.1 of
\cite{Saito}, Section~II (where $k$ was assumed to be finite).
\end{enumerate}
\end{remark}

\begin{cor} \label{purity}
Suppose that $h:\wh Y \to \wh X$ is a finite morphism of projective
$T$-curves, with $\wh X$ regular and connected, and with $\wh Y$ normal.  Then $h$ is a split
cover if and only if it is split over every  codimension one point of $\wh X$.
\end{cor}

\begin{proof}
The forward direction is trivial.  For the converse, we want to show that $h$
is split over every closed point.  By Purity of Branch Locus, $h$ is \'etale;
and hence every fiber is reduced and has separable residue field extension.
Any closed point $Q$ of $\wh X$ has codimension two, and by regularity there is a system of
local uniformizing parameters at $Q$.  Consider the zero locus of
one of these parameters, and let $Z$ be the irreducible component of this locus that passes
through $Q$.  Then $h$ splits over the generic point of the regular curve $Z$
by hypothesis, since this point has codimension one.  Hence there is no
residue field extension over the closed point $Q$ and thus $h$ is split
there.  
\end{proof}

We now return to the situation of Notation~\ref{notn_patches}. However, when
choosing a finite subset $\mc P$ of the closed fiber $X$, we now strengthen
our hypotheses slightly:

\begin{hyp} \label{notn_branched}
Under Notation~\ref{notn_patches}, assume in addition that the finite set $\mc
P \subset X$ contains all the closed points at which $X$ is not unibranched.
In particular, it contains all points where an irreducible component intersects
itself.
\end{hyp}

We then have the following analog of Proposition~\ref{split_points}:
\begin{cor} \label{split_patch}
In the situation of Hypothesis~\ref{notn_branched}, suppose that $h:\wh Y \to
\wh X$ is a finite morphism of normal projective $T$-curves, with $\wh X$
connected. 
Then $h$ is a split cover if and only if it is split over every $F_U$ (for $U \in
\mc U$) and every $F_P$ (for $P \in \mc P$).
\end{cor}

\begin{proof}
In the forward direction, $h$ is split over the generic point of each $U \in
\mc U$ (this being a codimension one point of $\wh X$), and $h$ is \'etale.
Let $Y$ be the closed fiber of $\wh Y$; and for any choice of $U \in \mc U$
let $V \subseteq Y$ be a connected component of $h^{-1}(U)$.  If $V$ is
reducible, then each irreducible component of $V$ must meet some other
component at some closed point $Q$, which is thus not unibranched.  But since
$h$ is \'etale, $h(Q) \in U$ is also not unibranched.  But $\mc P$ contains
all non-unibranched points by Hypothesis~\ref{notn_branched}, and so such a point cannot lie in
$U$.  This is a contradiction.  So actually $V$ is irreducible, and has a
unique generic point.
By splitness, the residue field at that generic point is a trivial field
extension of the residue field at the generic point of $U$.  Thus $V \to U$ is
a connected finite \'etale cover of degree one, and hence an isomorphism.
This shows that $h^{-1}(U)$ is a disjoint union of copies of $U$; i.e.\ $h$
splits over $U$.  By the observation before Proposition~\ref{split_points},
$h$ is split over $F_U$.  Also, $h$ is split over each $F_P$ by  Proposition~\ref{split_points}.

In the converse direction, since $h$ is split over each $F_U$ it is split over
each $\wh R_U$ (since $\wh Y$ is normal), and hence over $\wh R_Q$ for every
(not necessarily closed) point $Q \in U$.  But every point of $X$ lies either
in $\mc P$ or in some $U \in \mc U$.  So $h$ is split by
Proposition~\ref{split_points}.
\end{proof}

Notice that Proposition~\ref{split_patch} fails without the extra assumption in
Hypothesis~\ref{notn_branched}.  For example, take a non-trivial split cover of
a rational nodal curve, if the nodal point is not in the set $\mc P$.

Let $G$ be a finite constant group.  Given a $G$-Galois \'etale cover $\wh Y
\to \wh X$ with $\wh X$ connected, its generic fiber is a $G$-torsor $\T = \Spec(A)$ over the
function field $F$ of $\wh X$, and
$\wh Y$ is the normalization of $\wh X$ in $A$.  Moreover $\wh Y \to \wh X$ is
split over a field extension $E/F$ if and only if the $G$-torsor $\T$ becomes trivial over $E$.
Conversely, given a $G$-torsor $\T = \Spec(A)$ over~$F$, the normalization
$\wh Y \to \wh X$ of $\wh X$ in $A$ is a finite morphism of normal projective
$T$-curves, with a $G$-action that extends that of its generic fiber $\T$ and
satisfies $\wh Y/G = \wh X$.  Thus Proposition~\ref{split_points} and
Corollary~\ref{split_patch} yield: 

\begin{cor} \label{Sha_P-finite_gp}
Let $\wh X$ be a connected normal projective $T$-curve with function field $F$.
Then for any finite constant group $G$, and any $\mc P\subset X$ satisfying Hypothesis~\ref{notn_branched}, 
the subsets $\Hom(\pi_1^\s(\wh X),G){/\!\sim}$, $\Sha_X(F,G)$, and $\Sha_{\mc P}(F,G)$
of $H^1(F,G)$ each
classify the $G$-Galois finite split covers of $\wh X$, and are thus equal.
\end{cor}

Our aim is to extend the statements of Corollary~\ref{Sha_P-finite_gp} to
arbitrary rational linear algebraic groups. Using the results of
Section~\ref{sec_rational}, we immediately obtain:

\begin{cor} \label{Sha_P-pi}
Under Hypothesis~\ref{notn_branched}, if $G$ is a rational linear algebraic group then
$\Sha_{\mc P}(F,G)$ is in natural bijection with the
pointed set $\Hom(\pi_1^\s(\wh X),G/G^0)/\!\sim$.  Hence $\Sha_{\mc P}(F,G)$ is
independent of the choice of $\mc P$.
\end{cor}

\begin{proof}
By Theorem~\ref{sha_inj}, the natural map $\Sha_{\mc P}(F,G) \to \Sha_{\mc
  P}(\wh X,G/G^0)$ is a bijection.  Since $G$ is rational, its quotient
$G/G^0$ is a constant finite group.
Corollary~\ref{Sha_P-finite_gp} then implies that $\Sha_{\mc P}(F,G)$ is
in natural bijection with $\Hom(\pi_1^\s(\wh X),G/G^0){/\!\sim}$.

The second assertion follows from the fact that if $\mc P \subseteq \mc P'$
then the bijections with $\Hom(\pi_1^\s(\wh X),G/G^0)/\!\sim$ are compatible
with the inclusion $\Sha_{\mc P}(F,G) \subseteq \Sha_{\mc P'}(F,G)$.
\end{proof}

To extend this result to $\Sha_X(F,G)$, thus generalizing
Corollary~\ref{Sha_P-finite_gp}, we first prove two preliminary results.

\begin{prop}\label{approx}
Let $\wh X$ be a normal projective curve over a complete discrete valuation
ring $T$, and let $\eta$ be the generic point of an irreducible component
$X_0$ of the closed fiber $X$.  Let $H$ be a variety over $F$.  If $H(F_\eta)$
is non-empty, then so is $H(F_U)$ for some non-empty affine open subset $U
\subset X_0$ that does not meet any other irreducible component of $X$.
\end{prop}

\begin{proof}
Since $\wh X$ is normal and $\eta$ is a point of codimension one, its local ring $R_\eta$ is a discrete valuation ring, say with uniformizer $s$.
The given $F_\eta$-point on $H$ lies
on an affine open subset $H'$ of $H$.  Here $H'$ is given in affine $n$-space
by polynomials $f_j(x_1,\dots,x_n) \in F[x_1,\dots,x_n]$, for $j=1,\dots,m$.
Let the coordinates of the $F_\eta$-point be $(\hat a_1,\dots,\hat a_n)$, with
$\hat a_i \in F_\eta$.  For some integer $c$, the elements $s^c \hat a_i$ all
lie in the complete local ring $\wh R_\eta$.  Setting $g_j(x_1,\dots,x_n) =
f_j(s^{-c}x_1,\dots,s^{-c}x_n)$ and replacing the polynomials $f_j$ by the
polynomials $g_j$, we may assume that each $\hat a_i$ lies in $\wh R_\eta$.
Also, multiplying the polynomials $g_j$ by appropriate powers of $s$, we may
assume that each $g_j$ lies in $R_\eta[x_1,\dots,x_n]$.

Since $T$ is a complete discrete valuation ring, it is excellent by
\cite[Scholie~7.8.3(iii)]{EGA4.2}; and since $\wh X$ is of finite type over
$T$, it follows from \cite[Scholie~7.8.3(ii)]{EGA4.2} that the local ring
$R_\eta$ is also excellent.  Hence the Artin Approximation Theorem
(\cite{artin}, Theorem~1.10) applies.  Thus there exist
elements $a_1,\dots,a_n$ of the henselization $\tilde R_\eta$ of $R_\eta$
that satisfy the polynomials $g_j$ and are congruent to the elements $\hat
a_i$ modulo $s$.  Since the henselization is a limit of \'etale neighborhoods,
the elements $a_i$ all lie in some \'etale neighborhood of $\eta$, i.e.\ in
some \'etale $R_\eta$-algebra $S$ that is contained in $\tilde R_\eta$.  This
containment defines a section of $\Spec(S) \to \wh X$ over $\eta$.  Since
$R_\eta$ is the local ring of $\wh X$ at $\eta$, and since $\Spec(S) \to \wh
X$ is of finite type, there exists an affine neighborhood $\Spec(A)$ of $\eta$
in $\wh X$ together with an \'etale $A$-algebra $S_0 \subset S$ such that $S_0
\otimes_A {R_\eta}$ is isomorphic to $S$ as an $R_\eta$-algebra, and which
contains $a_1,\dots,a_n$.  The section over $\eta$
defines a rational section over $X_0$ and hence a section over some
affine open neighborhood $U \subset X$ of $\eta$ in $X$.
After shrinking $U$ we may assume that it is contained in the subset of $X_0$ consisting of points that do not meet any other component of $X$.
After shrinking
$\Spec(A)$ we may assume that $U$ is its closed fiber.

Since $\Spec(S_0) \to \Spec(A)$ is \'etale, \cite[Lemma~4.5]{HHK} implies that the section over $U$ extends to a section over $\Spec(\wh
R_U)$.  That is, the inclusion of $S_0$ into $\tilde R_\eta \subset \wh
R_\eta$ defines an inclusion of $S_0$ into $\wh R_U$, and hence into
$F_U$.  The $S_0$-valued point $(a_1,\dots,a_n)$ of $H' \subseteq H$ thus
determines an $F_U$-point of $H$, as desired.
\end{proof}

\begin{cor} \label{Sha_union}
Under Notation~\ref{notn_fields}, let $G$ be any linear algebraic group over
the function field $F$.  Then $\Sha_X(F,G) = \bigcup_{\mc P} \Sha_{\mc
  P}(F,G)$, where the union ranges over all finite sets $\mc P$ of closed
points of $\wh X$ that satisfy Hypothesis~\ref{notn_branched}.
\end{cor}

\begin{proof}
Since each $\Sha_{\mc P}(F,G)$ is contained in $\Sha_X(F,G)$, it
suffices  to show that every element of $\Sha_X(F,G)$ lies in some
$\Sha_{\mc P}(F,G)$.  So consider a torsor representing an element of
$\Sha_X(F,G)$.
Let $\eta$ be the generic point of an irreducible component $X_0$ of $X$.
Then the given torsor is trivial over $F_\eta$ and so has an $F_\eta$-point.
By Proposition~\ref{approx}, the torsor has an $F_U$-point for some
affine Zariski open subset $U \subset X_0$ that does not meet any other
irreducible component of $X$.  Hence the given torsor is trivial over $F_U$.
In this way we obtain such $U$ on each
irreducible component $X_0$ of $X$.
Take $\mc U$ to be the set of these open subsets $U$, indexed by the
irreducible components of $X$; and take $\mc P$ to be the complement in $X$ of
the union of the sets $U$.  Then the given torsor represents an element in
$\Sha_{\mc P}(F,G)$.
\end{proof}

\begin{thm}\label{sha_fsplit}
Let $F$ be a one variable function field over the field of fractions of a
complete discrete valuation ring $T$, and let $\wh X$ be a normal model for
$F$ over $T$, with closed fiber $X$.  Let
$G$ be a rational linear algebraic group defined over $F$. Then
\[\Sha_{\wh X,X}(F,G) = \Sha_{\wh X,\mc P}(F,G)\]
for any set $\mc P$ that satisfies Hypothesis~\ref{notn_branched}. 
Both are in natural bijection, as pointed sets, with
$\Hom(\pi_1^\s(\wh X),G/G^0)/\!\sim$.
In particular, the natural map of pointed sets $\Sha_X(F, G)\to \Sha_X(F,G/G^0)$ is a bijection.
\end{thm}

\begin{proof}
Since $\Sha_{\wh X,\mc P}(F,G)$ is independent of the choice of $\mc P$ by
Corollary~\ref{Sha_P-pi}, the first assertion is immediate from
Corollary~\ref{Sha_union}. The second and third statements then follow by Corollary~\ref{Sha_P-pi} and Theorem~\ref{sha_inj}.
\end{proof}

\section{The Reduction Graph}\label{sec_reduction}
Split covers of a $T$-curve $\wh X$ can be understood in terms of a
combinatorial object, viz.\ the reduction graph $\Gamma$ associated to the
closed fiber of $\wh X$.  Using this, we obtain a more explicit description of
$\Sha_X(F, G)=\Sha_{\mc P}(F,G) $ in the case of rational linear algebraic groups.  This description
(Corollary~\ref{Sha_P-graph}) shows that $\Sha_X(F,G)$ is finite, and it provides a
necessary and sufficient condition for it to vanish and thus for the
corresponding local-global principle to hold (still under the assumption that
$G$ is rational). This condition is given in
terms of the rational group $G$ and the reduction graph $\Gamma$.

We saw in Section~\ref{setup} how a finite subset of points ${\mc P}$
as in Notation~\ref{notn_patches} yields a factorization inverse system of
fields, and in Section~\ref{patching} how such systems relate
to (multi-)graphs.
In the situation of Notation~\ref{notn_patches}, the associated graph $\Gamma =
\Gamma(\wh X,\mc P)$, 
which we call the {\em reduction graph} of $(\wh X,\mc P)$, 
has vertices consisting of the elements of $\mc P
\cup \mc U$, and edges consisting of the elements of $\mc B$.   The vertices
of an edge $\wp$ are $P$ and $U$, where $\wp$ is a branch of $X$ at $P \in \mc
P$ lying on the closure of $U \in \mc U$.  Thus each edge connects a vertex in
$\mc P$ to a vertex in $\mc U$.  Graphs with this property (namely that the vertices are
partitioned into two disjoint sets such that each edge has a vertex in each
set) are called {\em bipartite}.
We may regard the graph either as a
combinatorial object or (as we will do more often) as a topological object ---
viz.\ a one-dimensional simplicial complex.  Note that the graph is connected
as a topological space because $X$ is connected.

\begin{remark} \label{graph_remark}
\renewcommand{\theenumi}{\alph{enumi}}
\begin{enumerate}
\item \label{barycentric}
Classically, there is another (multi-)graph that is associated to the above situation
in the special case that $X$ has only ordinary double points, and $\mc P$ is
the set of double points.  Namely, the edges are given by the points of $\mc
P$; the vertices are given by the irreducible components of the closed fiber;
and the vertices of a given edge $P \in \mc P$ are the components of $X$ that
pass through the point $P$.  (See~\cite[p.~86]{DM},
\cite[Definition~3.17]{Liu}, and \cite[Definition~II.2.3]{Saito}.)
In this special case, it is easy to see that this graph is homotopic, as a
topological space, to our graph $\Gamma$.
(In fact, our graph is the barycentric subdivision of the classical reduction
graph.)  Hence they have the same fundamental groups and cohomology.

\item \label{blow-up_tree}
If $(\wh X,\mc P)$ is as in Hypothesis~\ref{notn_branched} and if $\mc Q$ is a
  finite set of closed points that contains $\mc P$, then the graphs
  $\Gamma(\wh X,\mc P)$ and $\Gamma(\wh X,\mc Q)$ are homotopic, with the
  latter differing from the former by having additional terminal vertices 
  and edges corresponding to additional points and their branches.  
(This requires the additional assumption of
  Hypothesis~\ref{notn_branched}, since self-intersecting components can
  introduce loops into the graph.)
  Moreover if $\wh X' \to \wh X$ is a blow-up of $\wh X$ at a
  closed point at which $\wh X$ is regular,
  and if $\mc P' \subset \wh X'$ is a finite set of closed points over $\mc P$ that
contains the non-unibranched points of the closed fiber, then the graphs
$\Gamma(\wh X,\mc P)$ and $\Gamma(\wh X',\mc P')$ are homotopic.
\end{enumerate}
\end{remark}

Under Hypothesis~\ref{notn_branched}, if $h:\wh Y \to \wh X$ is a split cover
with closed fiber $Y$, then we may consider the finite set $
h^{-1}(\mc P) \subset Y$.  In this situation, $(\wh Y, h^{-1}(\mc P))$ also
satisfies Hypothesis~\ref{notn_branched}, and we obtain a graph $
\Gamma':=\Gamma(\wh Y,h^{-1}(\mc P))$.  There is a natural map of graphs
$\Gamma' \to  \Gamma $, and this is a finite covering space.  Conversely,
every finite covering space of $\Gamma$ has the structure of a graph, by
taking as the set of vertices the inverse image of the vertices of~$\Gamma$.

\begin{prop} \label{graph_corresp}
With $(\wh X,\mc P)$ as in Hypothesis~\ref{notn_branched}, the above
association defines a lattice isomorphism between  (connected) split covers
of $\wh X$ and (connected) finite covering spaces of $\Gamma = \Gamma(\wh
X,\mc P)$, which preserves the degree of the covers.
\end{prop}

\begin{proof}
We show that these objects are each classified by the same combinatorial data.

By Proposition~\ref{map_to_line}, there is a finite morphism $\wh X \to \mbb P^1_T$ such that $\mc P = f^{-1}(\infty)$.  
So by
\cite[Theorem~7.1(iv)]{HH:FP} in the context of \cite[Theorem~6.4]{HH:FP},
giving a finite separable commutative $F$-algebra $A$ is equivalent to giving its base
change to each of the fields $F_P$ (for $P \in \mc P$) and $F_U$ (for $U \in
\mc U$) together with isomorphisms between the algebras that they induce over
the fields $F_\wp$ (for $\wp \in \mc B$ a branch at $P$ lying on the closure
of~$U$).
But giving such an $F$-algebra $A$ is also equivalent to giving
a normal projective $T$-curve $\wh Y$ together with a finite generically separable morphism to $\wh X$ (where $A$ is the generic fiber of $\wh Y$, and $\wh Y$ is the normalization of $\wh X$ in $A$).  
Thus by Proposition~\ref{split_patch},
giving a split cover of $\wh X$ of degree
$n$ is equivalent to giving an $F_{\wp}$-isomorphism of $(\prod_{i=1}^n F_\xi)
\otimes_{F_\xi} F_\wp$ with $\prod_{i=1}^n F_\wp$ for each pair $\xi
\in \mc P \cup \mc U$ and $\wp \in \mc B$
such that $\wp$ is a branch at $\xi$
(if $\xi \in \mc P$) or on the closure of $\xi$ (if $\xi \in \mc U$).  That
is, we are given such an isomorphism for each pair $(\xi,\wp)$ such that
$\xi$ is a vertex of the edge $\wp$ in the reduction graph $\Gamma$.  This is the same as giving
a family of permutations $\sigma_{\xi,\wp} \in S_n$ indexed by such pairs, up
to equivalence coming from renumbering the $n$ copies of each $F_\xi$ and of
each $F_\wp$. That is, a family of permutations
$\{\sigma_{\xi,\wp}\}_{\xi,\wp}$ is equivalent to the family $\{\tau_\xi^{-1}
\sigma_{\xi,\wp} \tau_\wp\}_{\xi,\wp}$ for each choice of permutations
$\tau_\xi, \tau_\wp \in S_n$ indexed by the elements $\xi \in \mc P \cup \mc
U$ and $\wp \in \mc B$.

Meanwhile, to give a covering space of $\Gamma$ having degree $n$ is to give a
graph $\Gamma'$ whose vertex set is $(\mc P \cup \mc U) \times \{1,\dots,n\}$
and whose edge set is $\mc B \times \{1,\dots,n\}$.  Here $\Gamma'$ is
determined by specifying, for each edge $(\wp,i)$, the corresponding vertices
$(P,j)$ and $(U,\ell)$, where $P,U$ correspond to the vertices of the edge $\wp$
in $\Gamma$, and where $1 \le i,j,\ell \le n$.  For each such pair $(\wp,P)$,
the map $i \mapsto j$ is an element of $S_n$; and similarly for $(\wp,U)$ and
$i \mapsto \ell$.  Again we have equivalence corresponding to renumbering the
indices.
 Associating to each split cover of $\wh X$ the covering space of $\Gamma$ with
the same set of permutations thus gives a bijection between the split covers
of $\wh X$ and the covering spaces of $\Gamma$.  This is in fact the same map
$\wh Y \mapsto \Gamma'$ given before the proposition, by the definition of the
graph associated to the closed fiber.

A split cover $\wh Y \to \wh X$ of degree $n$ factors through a split cover
$\wh Y' \to \wh X$ of degree $m$ if and only if for some representative
families of permutations  $\{\sigma_{\xi,\wp}\}_{\xi,\wp}$ and
$\{\sigma'_{\xi,\wp}\}_{\xi,\wp}$ for these covers there is an $n/m$-to-one
map $\phi:\{1,\dots,n\}\to\{1,\dots,m\}$ such that
$\phi\sigma_{\xi,\wp}=\sigma'_{\xi,\wp}\phi$ for all $(\xi,\wp)$.  The same is
true for the covering spaces of $\Gamma$.  So the above bijection of lattices
preserves the property of one cover dominating another.  As a result, it
preserves meet and join in the lattices, and so it is a lattice isomorphism.

Finally, a split cover of $\wh X$, or a covering space of $\Gamma$, is
connected if and only if the set of associated permutations is transitive.  So
the above bijection carries connected split covers to connected covering
spaces.
\end{proof}

\begin{remark}
\renewcommand{\theenumi}{\alph{enumi}}
\begin{enumerate}
\item 
Since Proposition~\ref{graph_corresp} concerns finite covers, it could also
have been proven using other forms of patching (e.g.\ formal patching), rather
than using patching over fields.  See the introductory comments in Section~7.2
of~\cite{HH:FP}.
\item
In the special case that $\wh X$ is regular, Proposition~\ref{graph_corresp} can also be proven using Remark~\ref{split_remark}(\ref{bijection}).
Namely, after blowing up $\wh X$ at points of $\mc P$, we obtain a model with the property that no two branches at any point can lie on the same irreducible component of the closed fiber.  
Since $\wh X$ is regular, we may replace $\wh X$ by this model, using the last part of Remark~\ref{graph_remark}(\ref{blow-up_tree}).  
Given a finite covering space $\Gamma' \to \Gamma$, it is now easy to construct a cover of $X$ that is split over every point and which induces $\Gamma' \to \Gamma$, by constructing this cover Zariski locally over $X$.  
As a result we obtain the inverse of the association considered in Proposition~\ref{graph_corresp}, via Remark~\ref{split_remark}(\ref{bijection}).
\end{enumerate}
\end{remark}

Recall that $\pi_1^{\s}(\wh X)$ is the inverse limit of the Galois groups of Galois split covers of $\wh X$ as introduced in Section~\ref{sec_splitcovers}. Directly from Proposition~\ref{graph_corresp}, we obtain:
\begin{cor} \label{graph_pi1}
Under Hypothesis~\ref{notn_branched}, $\pi_1^{\s}(\wh X)$ is naturally
isomorphic to the profinite completion of $\pi_1(\Gamma(\wh X,\mc P))$.
Thus it is a free profinite group on finitely many generators.
\end{cor}

Here the last assertion follows from the fact that the fundamental group of
any graph is a free group on finitely many generators.

As a consequence of Corollary~\ref{graph_pi1}, $\pi_1^{\s}(\wh X)$ is trivial if and only if the graph
$\Gamma(\wh X,\mc P)$ is a tree.
As another consequence, the fundamental group of this graph is independent of
the choice of the set $\mc P$, and may be denoted by $\pi_1(\Gamma(\wh X))$.  

\begin{cor} \label{Sha_P-graph}
Let $G$ be a rational linear algebraic group. Then
$\Sha_X(F,G)$ is finite, and is in natural
bijection with the pointed set $\Hom(\pi_1(\Gamma(\wh X)),G/G^0)/\!\sim$.  
It is trivial if and only if $G$ is connected or
$\pi_1(\Gamma(\wh X))=1$.
Moreover, for any $\mc P$ as in Hypothesis~\ref{notn_branched}, the same
holds for $\Sha_{\mc P}(F,G)$.
\end{cor}

\begin{proof}
Theorem~\ref{sha_fsplit} and Corollary~\ref{graph_pi1} together provide the
asserted natural bijection,
using that $\Hom(\Pi,H)$ is in natural bijection with $\Hom(\wh \Pi,H)$ for any finite group $H$ and any discrete group $\Pi$ with profinite completion $\wh \Pi$. 
The set $\Hom(\pi_1(\Gamma(\wh X)),G/G^0)/\!\sim$ is finite because $G/G^0$ is finite and $\pi_1(\Gamma(\wh X))$ is finitely generated. This gives the first assertion.

The set $\Hom(\pi_1(\Gamma(\wh X)),G/G^0)/\!\sim$ is
trivial exactly when the set $\Hom(\pi_1(\Gamma(\wh X)),G/G^0)$ is
trivial, since the only element conjugate to the identity is itself. As
observed above, $\pi_1(\Gamma(\wh X))$ is a free group of finite rank.
But the set of homomorphisms from such a free group to a finite (constant)
group is trivial if and only if the free group or the constant group is
trivial. So the second assertion now follows from the first assertion.
The last assertion follows again from Theorem~\ref{sha_fsplit}.
\end{proof}

\begin{remark}
\renewcommand{\theenumi}{\alph{enumi}}
\begin{enumerate}
\item 
In the special case that $\wh X$ is smooth over $T$, the reduction graph is
trivial and so Corollary~\ref{Sha_P-graph} says that $\Sha_{\wh X,\mc P}(F,G)$
is trivial (i.e.\ the associated local-global principle is satisfied) even if
$G$ is a disconnected rational group. The corresponding simultaneous
factorization result also appeared in \cite[Theorem~3.4]{HHK}, where
connectivity was not in fact needed or used.
\item
Although the group $\pi_1^{\s}(\wh X)$ is finitely generated, the group $\pi_1(\wh X)$ (of which it is a quotient) in general is not.  In fact, that group is isomorphic to $\pi_1(X)$, which has $\Gal(k)$ as a quotient.
\end{enumerate}
\end{remark}
\section{Split Extensions} \label{sec_splitexten}

In this section we obtain a Galois-theoretic description of 
$\pi_1^\s(\wh X)$, given just in terms of $F$ and the discrete valuations on $F$.  
For this, we need to require that the model $\wh X$ is regular; such models always exist (see \cite{Abh} or \cite{Lip}).  As a consequence of this description, 
$\pi_1^\s(\wh X)$ is independent of the choice of regular model $\wh X$ of $F$; and hence so is $\Sha_{\wh X,X}(F,G)$, in the case that $G$ is rational.  Indeed, for such models 
$\wh X$ and groups $G$, we obtain a description of $\Sha_{\wh X,X}(F, G)$ in
terms of the Galois cohomology of $F$ (Theorem~\ref{sha_fs}).  

Our approach will rely on the notion of {\em split extensions}; see the discussion
beginning just before Proposition~\ref{split_equiv}.
That proposition gives a characterization of split
extensions of $F$ in terms of split covers of a regular model. For this, we
need some preliminary results about valuations, which will also be used in 
Section~\ref{sec_valuations}.

\begin{prop} \label{val_ring}
Let $R$ be a complete local domain that is not a field, and let $v$ be a
discrete valuation on the fraction field of $R$.  Then $R$ is contained in the
valuation ring of $v$.
\end{prop}

\begin{proof}
Since $R$ is not a field, its maximal ideal $\mathfrak m$ contains a non-zero
element $m$.  For any non-zero $r \in R$ and any positive integer $i$, the
element $r^i m$ lies in $\mathfrak m$.

So by Hensel's Lemma, $1 + r^i m$ is an $n$-th power in $R$ for every $n$ that
is not divisible by the residue characteristic of $R$.  Thus $v(1 + r^i m)$ is
divisible by arbitrarily large integers in the value group of $v$; and so $v(1
+ r^i m) = 0$.  Thus $1 + r^i m$ lies in the valuation ring of $v$, and hence
so does $r^i m$.  Thus $iv(r)+v(m) = v(r^i m) \ge 0$ for every positive
integer $i$.   This implies that $v(r) \ge 0$ and hence $r$ is in the
valuation ring of $v$. \end{proof}

\begin{cor} \label{val-ring_contains_dvr}
Let $F$ be a field that contains a complete discrete valuation ring $T$, and
let $v$ be any discrete valuation on $F$.  Then the valuation ring of $v$
contains $T$.
\end{cor}

\begin{proof}
If the restriction of $v$ to the fraction field $K$ of $T$ is trivial, then
the valuation ring of $v$ contains $K$ and hence contains $T$.  Alternatively,
if the restriction of $v$ to $K$ is non-trivial, then this restriction is a
discrete valuation on $K$, and the assertion follows from
Proposition~\ref{val_ring} by taking $R=T$.
\end{proof}

Let $Z$ be an integral scheme and let $v$ be a discrete valuation on the
function field $F$ of $Z$, with valuation ring $R_v \subset F$.  Recall that a
point $Q$ of $Z$ is called the \textit{center} of $v$ if $R_v$ contains the
local ring $\mc O_{Z,Q}$ of $Q$ on $Z$, and if the maximal ideal ${\mathfrak
  m}_Q$ of $\mc O_{Z,Q}$ is the contraction of the maximal ideal ${\mathfrak
  m}_v$ of $R_v$.  This is equivalent to the condition that the morphism
$\Spec(F) \to \Spec(\mc O_{Z,Q})$ factors through the morphism $\Spec(F) \to
\Spec(R_v)$.  Since such a point $Q \in Z$ is the image of the closed point of
$\Spec(R_v)$ under the natural morphism $\Spec(R_v) \to Z$, the point $Q$ is unique if it exists.  Note that since the
valuation $v$ is non-trivial, $R_v$ is strictly contained in $F$ and hence so
is $\mc O_{Z,Q}$; thus the center of $v$ cannot be the generic point of $Z$.

\begin{lemma} \label{centers}
Let $T$ be a complete discrete valuation ring, and let $\wh X$ be a connected regular
projective $T$-curve with function field $F$.  Then
every discrete valuation $v$ on $F$ has a center $Q$ on $\wh X$, where $Q$ is
not the generic point of $\wh X$.  Moreover the completion $F_v$ contains $F_Q$.
\end{lemma}

\begin{proof}
By Corollary~\ref{val-ring_contains_dvr}, the valuation ring $R_v \subset F$
of $v$ contains $T$; and this inclusion corresponds to a morphism $\Spec(R_v)
\to \Spec(T)$.  Since $\wh X$ is proper over $\Spec(T)$, it follows from the
valuative criterion of properness~\cite[Theorem~II.4.7]{Hts} that there is a
morphism $j:\Spec(R_v) \to \wh X$ such that the following diagram commutes:
$$
\xymatrix{
\Spec(F) \ar[r] \ar[d]
& \wh X \ar[d]
\\
\Spec(R_v) \ar[r] \ar@{-->}[ur]
& \Spec(T)
}
$$
Thus $\Spec(R_v) \to \wh X$ and hence $\Spec(F) \to \wh X$ factors through
$\Spec(\mc O_{\wh X,Q})$, where $Q \in \wh X$ is the image under $j$ of the
closed point of $\Spec(R_v)$.  So $Q$ is the center of $v$ (which is not the
generic point of $\wh X$, as observed before the proposition).
In particular, we have an inclusion $\mc O_{\wh X,Q} \subseteq R_v$, with the
maximal ideal of $R_v$ contracting to that of $\mc O_{\wh X,Q}$.  This in turn
yields an inclusion $\wh{\mc O}_{\wh X,Q} \subseteq \wh R_v$ between their
completions; and taking fraction fields yields the desired inclusion.
\end{proof}

The next result strengthens the last assertion of Lemma~\ref{centers} to show that for each discrete valuation $v$, there is in fact a point $P$ on the closed fiber $X$ (rather than just on $\wh X$) for which $F_P\subseteq F_v$. We first introduce some notation.
If $Q$ is a codimension one point of a noetherian connected regular scheme, then it
defines a discrete valuation $v$ on the function field $F$, whose ring of
integers $R_v$ is the local ring $R_Q$ at the point $Q$ (\cite{Hts}, p.~130).
The completion $\wh
R_Q$ of $R_Q$ with respect to its maximal ideal $\mathfrak m_Q$ is the
$v$-adic completion $\wh R_v$ of $R_v$, and the fraction field $F_Q$ of $\wh
R_Q$ is the fraction field of $\wh R_v$, viz.\ the $v$-adic completion $F_v$
of $F$.

\begin{prop}  \label{field_containment}
Under the hypotheses of Lemma~\ref{centers}, for every discrete
valuation $v$ on $F$ there is a point $P$ on the closed fiber of $\wh X$ such
that $F_P \subseteq F_v$.
\end{prop}

\begin{proof}
By Lemma~\ref{centers}, the valuation $v$ has a center $Q$, which is not
the generic point of $\wh X$.  If $Q$ lies on the closed fiber $X$ of $\wh X$,
then we may take $Q=P$.  Otherwise, $Q$ lies on the general fiber of $\wh X$,
which is a regular projective curve over the fraction field $K$ of $T$.
The point $Q$ has codimension one, and $v$ is defined by $Q$, so $F_Q = F_v$.
The residue field at $Q$ is a finite extension $K'$ of $K$.  Here $K'$ is a
complete discretely valued field, and is the function field of the closure of
$Q$ in $\wh X$.  Let $P \in X$ be the closed point of $\wh X$ corresponding to
the maximal ideal of the valuation ring $T' \subset K'$.  It remains to show
that $F_P \subseteq F_Q$.

The point $P$ is in the closure of $Q$, and so $R_P$ is contained in $R_Q$.
Since $\wh X$ is regular, $R_P$ is a regular local ring.  So the height one
prime ideal $I \subset R_P$ defining $Q$ is principal; say $I = (f)$.  We
claim that $\wh R_P$, the completion of $R_P$ with respect to its maximal
ideal $\mathfrak m$, is also the $I$-adic completion of $R_P$.  If this is
shown, then the inclusion $R_P \subset R_Q$ induces an inclusion $\wh R_P
\subset \wh R_Q$ on the $I$-adic completions, and hence induces the desired
inclusion $F_P \subseteq F_Q$.

The ideal $\mathfrak m \subset R_P$ is the radical of the ideal $(f,t)$, where
$t$ is the uniformizer of $T$; so the
$\mathfrak m$-adic and $(f,t)$-adic topologies on $R_P$ are the same.
For any positive integer $n$, the ring $R_P/I^n$ is $t$-adically complete
because this ring is a finite module over the $t$-adically complete ring $T$.
But a sequence in $R_P/I^n$ is $t$-adically Cauchy if and only if it is
$(f,t)$-adically Cauchy, because $f$ is nilpotent in $R_P/I^n$.  Hence
$R_P/I^n$ is $(f,t)$-adically complete.  Thus so is the inverse limit of these
rings, which is the $I$-adic completion $\wh{(R_P)_I}$ of $R_P$.  Since $I
\subset \mathfrak m$, the $I$-adic completion of $R_P$ is contained in the
$\mathfrak m$-adic completion $\wh R_P$ of $R_P$.
But $\wh R_P$ is the smallest $\mathfrak m$-adically complete ring containing
$R_P$; so the containment $\wh{(R_P)_I} \subseteq \wh R_P$ is an equality.
This proves the claim and hence the result.
\end{proof}

\begin{prop} \label{restriction}
Let $F$ be a function field in one variable over a complete discretely valued
field $K$ with valuation ring $T$. Let $\wh X$ be a regular $T$-curve with function field $F$.
If $P$ is a closed point of $\wh X$ and
$v$ is a discrete valuation on $F_P$ then the restriction of $v$ to $F$ is a
discrete valuation on $F$.
\end{prop}

\begin{proof} To prove the assertion we need to show that the restriction of
  $v$ to $F$ is non-trivial; i.e.\ $v(f) \ne 0$ for some $f \in F^\times$.

By Proposition~\ref{map_to_line},
there is a finite morphism $f:\wh X \to \mbb P^1_T$ taking $P$ to the point
$P'$ at infinity on the closed fiber $\mbb P^1_k$ of $\mbb P^1_T$.  Thus $F$
is a finite extension of the function field $F' = K(x)$ of $\mbb P^1_T$.  Let
$v'$ be the restriction of $v$ to $F'_{P'}$, the fraction field of the
complete local ring of $\mbb P^1_T$ at $P'$.
Now $F_P$ is a finite field extension of $F'_{P'}$;
so~\cite[Proposition~VI.8.1.1]{Bo:CA} implies that the value group of $v'$
has finite index in that of $v$.  Hence $v'$ is a (non-trivial) discrete
valuation on $F'_{P'}$.

The ring $\wh R_P$ is contained in the valuation ring $R_v \subset F_P$ of $v$
by Proposition~\ref{val_ring}, taking $R = \wh R_P$ there.  Thus $\wh R'_{P'}
= \wh R_P \cap F'_{P'}$ is contained in $R_{v'} := R_v \cap F'_{P'}$, the
valuation ring of $v'$.  Hence $v'(u)=0$ for every unit $u \in \wh R'_{P'}$.

Let $a \in R_{v'} \subset F'_{P'}$ be a uniformizer for $v'$ and let $\eta'$ be the generic point of $\mbb P^1_k$.  Thus $v'(a)>0$.
By~\cite[Proposition~5.6]{HH:FP} (taking $\wh R = \wh R'_{\{P'\}}$, $\wh R_1 = \wh R'_{P'}$, $R_2 = \wh R'_{\eta'}$, and $\wh 
R_0 = \wh R'_{\wp'}$), 
and writing $a \in F'_{P'}$ as a ratio of elements in $\wh R'_{P'}$,
there exist $u_1 \in 
\wh R_{P'}'^\times$ and $a_1 \in F'_{\{P'\}}$ such that $a=u_1a_1$.  
By~\cite[Corollary~4.8]{HH:FP}, there exist $u_2  \in \wh R_{\{P'\}}'^\times$ and $b \in F'$ such that $a_1=u_2b$.  
Thus $a=ub$ with $u = u_1u_2 
\in \wh R_{P'}'^\times$.
But $v'(u)=0$; so $v'(b)>0$.  Thus the
restriction of $v'$ to $F'$ is non-trivial.  Hence so is the restriction of
$v$ to $F$.
\end{proof}

We say that a finite extension $E/F$ of fields is {\em split} if $E\otimes_F
F_v$ is a product of copies of the completion $F_v$ of $F$ with respect to
$v$, for every  discrete valuation $v$ of $F$ (i.e.\ valuation on $F$ with
value group isomorphic to $\mbb Z$).  Let $\FS$ be the union of all split
extensions of $F$, inside a fixed algebraic closure of $F$.

\begin{prop}\label{split_equiv}
Let $F$ be a function field in one variable over a complete discretely valued
field $K$ with valuation ring $T$. Let $\wh X$ be a regular $T$-curve with function field $F$, and let $E/F$ be a finite field extension. Then the
following are equivalent:
\renewcommand{\theenumi}{\roman{enumi}}
\renewcommand{\labelenumi}{(\roman{enumi})}
\begin{enumerate}
\item\label{split_equiv_i}  $E/F$ is a split extension.
\item\label{split_equiv_ii} There is a split cover $\wh Y$ of $\wh X$ with function field $E$.
\end{enumerate}
Moreover, sending split covers to their function fields defines a lattice isomorphism between the connected finite split covers of $X$ and the finite split extensions of $F$.
\end{prop}

\begin{proof} If $E/F$ is a split extension, let $\wh Y$ be the normalization
  of $\wh X$ in $E$. For every codimension one point $Q$ on $\wh X$, the
  fraction field $F_Q$ of the complete local ring $\wh R_Q$ at $Q$ is the
  completion of $F$ at the discrete valuation $v$ defined by $Q$. By
  definition of split extensions, $E \otimes_F F_v$ is a product of copies
  of $F_v$. That is, $E/F$ is split over $F_Q$. Since $\wh X$ is regular (and connected),
  Corollary~\ref{purity} asserts that $\wh Y \to \wh X$ is a split cover.

Conversely, if $\wh Y \to \wh X$ is a split cover with function field $E$, let
$v$ be a discrete valuation on $F$.  By Lemma~\ref{centers} above, $v$
has a center $Q$ on $\wh X$ that is not the generic point, and $F_Q \subseteq
F_v$.  By definition of a split cover, $\wh Y \to \wh X$ splits over the point
$Q$ (since it is not the generic point of $\wh X$); and so by Hensel's Lemma
it splits over $\wh R_Q$ and thus over $F_Q$.  The containment $F_Q \subseteq
F_v$ implies that $\wh Y \to \wh X$ splits over $F_v$, and hence so does
$E/F$.

For the last assertion, it remains to show injectivity (on isomorphism classes).  This 
follows using that finite split covers are \'etale, together 
with the fact that a connected \'etale cover $\wh Y$ of $\wh X$ is the normalization of $\wh X$ in the function field of $\wh Y$.
\end{proof}

\begin{remark}\label{split_equiv_rk}
The proof of Proposition~\ref{split_equiv} shows that condition~(\ref{split_equiv_ii})
follows from an a priori weaker form of condition~(\ref{split_equiv_i}),
viz.\ it suffices to assume that $E/F$ is split over each discrete valuation of $F$ that is given by a codimension one point on the chosen regular model $\wh X$.  Since~(\ref{split_equiv_ii})
implies~(\ref{split_equiv_i}),
it follows that an extension $E/F$ is split if and only if it is split over $F_v$ for each $v$ in that smaller set of valuations.
\end{remark}

By taking Galois groups of the objects in Proposition~\ref{split_equiv} and passing to the limit, we obtain the following:

\begin{cor} \label{fs_free}
Let $F$ be a one-variable function field over the fraction field of a
complete discrete valuation ring.  Then $\FS/F$ is a Galois extension; and
for any regular  model $\wh X$ of $F$, there is a natural identification
\[\Gal(\FS/F) = \pi_1^\s(\wh X) \]
of free profinite groups on finitely many
generators.  Hence $\pi_1^\s(\wh X)$  and the fundamental group of the reduction graph are independent of the choice of the regular model $\wh X$ for $F$.
\end{cor}

\begin{proof}
Since split covers are \'etale, Proposition~\ref{split_equiv} implies that every finite split extension of $F$ is separable; and hence so is $\FS/F$.  
Also, $\FS$ is invariant under the absolute Galois group of $F$.  Thus the extension $\FS/F$ is Galois.  Moreover, under the lattice isomorphism given 
in the last part of Proposition~\ref{split_equiv}, each
Galois finite split cover of $\wh X$ is sent to a Galois finite split extension of $F$ with the same Galois group.
Hence the asserted identification of profinite groups follows.  The fact that these profinite groups are finitely generated and free then follows from Corollary~\ref{graph_pi1}, as does the last assertion of the present corollary. 
\end{proof}

Combining this with Corollary~\ref{Sha_P-graph} if $G$ is rational, we obtain
that $\Sha_{\wh X,X}(F,G)$ does not depend on the regular model $\wh X$ of $F$ (as $\FS$ does not).
To indicate this independence of the model, we write
$$\Sha_0(F,G):=\Sha_{\wh X,X}(F,G),$$
where $\wh X$ is any
regular model of $F$, and $X$ is its closed fiber.

\begin{lemma} \label{split_inclusions}
Let $\wh X$ be a regular projective curve over a complete discrete valuation ring $T$, with closed fiber $X$.
\renewcommand{\theenumi}{\alph{enumi}}
\renewcommand{\labelenumi}{(\alph{enumi})}
\begin{enumerate}
\item \label{point_inclusions}
For every point $P$ of $X$, there is an inclusion $\FS\hookrightarrow
F_P$ of $F$-algebras.
\item \label{patch_inclusions}
For any irreducible component $X_0 \subseteq X$, and any
non-empty affine open subset $U \subset X_0$ that does not meet any other
irreducible component of $X$, there is an inclusion $\FS \hookrightarrow F_U$
of $F$-algebras.
\end{enumerate}
\end{lemma}

\begin{proof}
(\ref{point_inclusions})
Let $\bar F$ be an algebraic closure of $F$.  The field $\FS \subset \bar F$ is a union of finite split
field extensions $F_i/F$, each of which is the function field of a
finite split cover $\wh X_i \to \wh X$ (by Proposition~\ref{split_equiv}).  Since this cover is split, it is also split over $F_P$ by
Proposition~\ref{split_points}.  
Hence $\Hom_F(F_i,F_P)$ is a non-empty finite set.  Thus
$\displaystyle\Hom_F(\FS,F_P) = \lim_\leftarrow \Hom_F(F_i,F_P)$ is also non-empty, being an inverse limit of non-empty finite sets.

(\ref{patch_inclusions})
Let $\mc P$ be the finite subset of $X$ that consists of the complement of $U$
in $X_0$ together with all the points of $X$ at which distinct irreducible
components of $X$ meet.  Let $\mc U$ be the set of connected components of the
complement of $\mc P$ in $X$.  Thus $U \in \mc U$.  The proof is now the same
as that of part~(\ref{point_inclusions}), except that
Corollary~\ref{split_patch} replaces Proposition~\ref{split_points}.
\end{proof}

\begin{thm}\label{sha_fs}
Let $G$ be a rational linear algebraic group defined over $F$. Then
\[\Sha_0(F,G)=H^1(\FS/F,G).\]
\end{thm}

\begin{proof}
Let $\iota: H^1(\FS/F,G)\rightarrow H^1(F,G)$ be the natural inclusion. By
Lemma~\ref{split_inclusions}(\ref{point_inclusions}), $\FS$ includes into
$F_P$ for each point $P$ of $X$.
Hence the image of $\iota$ is in $\Sha_0(F,G)$. It suffices to show that
$\iota$ surjects onto $\Sha_0(F,G)$.

Write $\bar G = G/G^0$ and let $\wh X$ be a regular model of $F$ over $T$.  To
show surjectivity, consider an element $\alpha\in \Sha_0(F,G)$, with image $\bar \alpha\in
\Sha_0(F,\bar G)$.
By Corollary~\ref{Sha_P-finite_gp}, $\bar\alpha$ corresponds to a $\bar G$-Galois split cover
$\wh Y \to \wh X$, 
which is a disjoint union of copies of a connected Galois split cover with function field $E/F$.
Write $\alpha_E$ for the image of $\alpha$ under
$\Sha_0(F,G) \to \Sha_0(E,G)$.

Since $E/F$ is Galois, $E \otimes_F E$ is a direct product of copies of $E$,
and hence the induced element $(\bar\alpha)_E \in \Sha_0(E,\bar G) \subseteq
H^1(E,\bar G)$ is trivial.  But this is the image of $\alpha_E$ under $\Sha_0(E,G)
\to \Sha_0(E,\bar G)$, and this map is bijective by Theorem~\ref{sha_fsplit}.
So $\alpha_E$ is trivial. Since $E \subseteq \FS$, it follows that $\alpha$ is in
fact in $H^1(\FS/F,G)$, as desired.
\end{proof}

\section{Local-global principles with respect to valuations} \label{sec_valuations}

In this section, we consider a local-global map with respect to the discrete valuations on a field $E$; this was previously considered in \cite{COP} and \cite{CGP}. Let $\Omega_E$ denote the set of equivalence classes of discrete valuations $v$ on $E$ (i.e., valuations with value group isomorphic to ${\mathbb Z}$), and let $E_v$ be the $v$-adic completion of $E$.
Given an algebraic group $G$ over $E$, we define
\[\Sha(E,G) := \ker\bigl(H^1(E,G) \to \prod_{v \in \Omega_E} H^1(E_v,G)\bigr).\]
Thus $\Sha(E,G)$ classifies $G$-torsors over $E$ that become trivial over each $E_v$.  Again, this is a pointed set, and is a group if $G$ is commutative.

We will focus on the case that $E$ is a one-variable function field $F$ 
over a complete discretely valued field $K$, and $G$ is a linear algebraic group over $F$. As before, the valuation ring of $K$ will be denoted by $T$.  And as in 
Section~\ref{sec_splitexten} we consider only 
regular models of $F$, i.e., regular projective $T$-curves $\wh X$ with function field $F$.
We will also consider the case that $E=F_P$, for $P$ a point on the closed fiber $X$ of such a model $\wh X$ of $F$.

For $F$, $\wh X$, $X$ as above, we will
study the relationship between $\Sha(F,G)$ and $\Sha_X(F,G)$, using this to obtain information about local-global principles for torsors in terms of discrete valuations.  (For now, we do \textit{not} assume that $G$ is rational, and so we
cannot conclude from Section~\ref{sec_splitexten}
that $\Sha_X(F,G)$ depends only on the function field $F$.)
Theorem~\ref{local-global_Sha_disconnected} gives a criterion for the equality of $\Sha(F,G)$ and $\Sha_X(F,G)$, and this is then used to obtain specific conditions guaranteeing this equality (Theorem~\ref{sha_conditions}).  In particular, Corollary~\ref{Sha_descrip} gives a criterion under which $\Sha(F,G)$ is finite and there is a necessary and sufficient condition for it to vanish.

\begin{remark} \label{valuations_on_models}
The above definition of $\Sha(F,G)$ was used in \cite{CPS}, and it was also
considered (without this notation) in \cite{COP}.  In \cite{CGP} and
\cite{BKG}, a slightly different version of $\Sha(F,G)$ was considered.
There, one takes only those discrete valuations that arise from codimension
one points on blow-ups of a given regular model.  In our situation, these two
sets of discrete valuations are the same if the residue field $k$ of $T$ is
algebraically closed; but more generally the sets are unequal.  In that case,
our $\Sha(F,G)$ is contained in the variant version, and it would be
interesting to know if this latter containment is always an equality.  As the
proofs below show, the results stated here about $\Sha(F,G)$ each hold with
either version, and some of the results hold even if one just considers the
discrete valuations associated to codimension one points on a fixed regular
model (as was the case with
Proposition~\ref{split_equiv}/Remark~\ref{split_equiv_rk}).
\end{remark}

\begin{prop} \label{sha_inclusion}
Let $F$ be a one-variable function field over a complete discretely valued field $K$, and let $G$ be a linear algebraic group over $F$.
Let $\wh X$ be a regular model of $F$, with closed fiber $X$.
Then
$\Sha_X(F,G) \subseteq \Sha(F,G)$ as subsets of $H^1(F,G)$.
\end{prop}

\begin{proof}
Let $\alpha \in H^1(F,G)$ be an element of $\Sha_X(F,G)$.  Then $\alpha$ has trivial image in $H^1(F_P,G)$ for every $P \in X$.  Let $v \in \Omega_F$.  By Proposition~\ref{field_containment}, there is some point $P$ on the closed fiber $X$ of $\wh X$ such that $F_P$ is contained in $F_v$.  Hence $\alpha$ has trivial image in $H^1(F_v,G)$.  Thus $\alpha$ is in $\Sha(F,G)$.
\end{proof}

\begin{prop} \label{loc-to-glob_sha_map}
Let $\wh X$ be a regular projective curve over a complete discrete valuation ring $T$, with function field F and closed fiber $X$.  Let $X_{(0)}$ be the set of closed points of $X$.  For any linear algebraic group $G$ over $F$,
the image of $\Sha(F,G)$ under the natural map
\[\phi:H^1(F,G) \to \prod_{P \in X_{(0)}} H^1(F_P,G)\]
is $\displaystyle {\prod_{P \in X_{(0)}}}{\!\!\!'} \ \Sha(F_P,G)$, the subset of the product in which all but finitely many entries are trivial.
\end{prop}

\begin{proof}
For any $P \in X_{(0)}$ and any $v \in \Omega_{F_P}$, the restriction $v_0$ of $v$ to $F$ is a discrete valuation in $\Omega_F$, by Proposition~\ref{restriction}.
If $\alpha$ is in $\Sha(F,G)$ then the image of $\alpha$ in $H^1(F_{v_0},G)$ is trivial; and hence so is its image in $H^1((F_P)_v,G)$.  But this is the same as the image  in $H^1((F_P)_v,G)$ of the $P$-component
of $\phi(\alpha)$.  Thus $\phi(\alpha)$ and hence $\phi(\Sha(F,G))$ are
contained in $\prod_{P \in X_{(0)}}\Sha(F_P,G)$.  

Let $X_1,\dots,X_n$ be the irreducible components of $X$, and let $\eta_i$ be the generic point of $X_i$.  Then $\eta_i$ is a codimension one point of $\wh X$, and so corresponds to a discrete valuation in $\Omega_F$.  Thus if $\alpha$ is an element of $\Sha(F,G)$, then the image of $\alpha$ is trivial in $H^1(F_{\eta_i},G)$.
That is, $\alpha$ corresponds to a $G$-torsor over $F$ that has an $F_{\eta_i}$-point.
By Proposition~\ref{approx},
it also has an $F_{U_i}$-point, for some affine open subset $U_i \subset X_i$
that does not meet any other component.  Thus
the image of $\alpha$ is trivial in $H^1(F_{U_i},G)$,
and hence also trivial in $H^1(F_P,G)$ for all $P \in U_i$.  Since
there are only finitely many points of $X$ that do not lie in any of the sets
$U_i$, it follows that $\phi(\Sha(F,G))$ is indeed contained in $\prod'_{P \in X_{(0)}} \Sha(F_P,G)$.

To prove the result we will show that $\Sha(F,G)$ surjects onto $\prod'_{P \in
  X_{(0)}} \Sha(F_P,G)$ under $\phi$.  So consider a finite set $S$ of closed points of $X$, and elements $\alpha_P \in \Sha(F_P,G)$ for $P \in S$.  For each $P \in X_{(0)}$ that is in the complement of $S$, let $\alpha_P$ be the trivial element of $\Sha(F_P,G)$.  We wish to find an element $\alpha \in \Sha(F,G)$ whose image in $\Sha(F_P,G)$ is $\alpha_P$ for all $P \in X_{(0)}$.

After enlarging $S$, we may assume that $S$ contains a point on each irreducible component $X_i$ of $X$, and includes all the points at which $X$ is not unibranched.
Since $\alpha_P$ is in $\Sha(F_P,G) \subseteq H^1(F_P,G)$, the restriction of $\alpha_P$ to $H^1(F_\wp,G)$ is trivial for any height one prime $\wp$ of $\wh R_P$.  In particular, this holds if $\wp$ is a branch of $X$ at $P$.

Let $U_i$ now be the complement of $S \cap X_i$ in $X_i$, and let
$\alpha_{U_i}$ be the trivial element of $H^1(F_{U_i},G)$.  Thus
$\alpha_{U_i}$ is an element of $\Sha(F_{U_i},G)$, and its restriction to
$H^1(F_\wp,G)$ is trivial for any branch $\wp$ at a point $P \in S \cap X_i$.
Thus $\alpha_{U_i}, \alpha_P$ induce isomorphic (trivial) elements of
$H^1(F_\wp,G)$, if $P \in S \cap X_i$ and $\wp$ is a branch at $P$ on the
closure of $U_i$.  Hence Theorem~\ref{6-term_sequence} applies; 
and so there is some $\alpha \in H^1(F,G)$ that induces $\alpha_P \in
H^1(F_P,G)$ for each $P \in S$ and induces $\alpha_{U_i} \in H^1(F_{U_i},G)$
for each $i$.  Since $\alpha_{U_i}$ is trivial, it follows that $\alpha$
induces the trivial element of $H^1(F_P,G)$ for each $P$ in the complement of
$S$. 

It remains to show that $\alpha$ lies in $\Sha(F,G)$.  So consider any $v \in \Omega_F$.  By Proposition~\ref{field_containment}, there is a point $P \in X$ such that $F_P$ is contained in $F_v$.  If $P$ is a codimension one point of $\wh X$, then $P = \eta_i$ for some $i$; and then the image of $\alpha$ in $H^1(F_{\eta_i},G)$ is trivial since its image in $H^1(F_{U_i},G)$ is trivial and since $F_{\eta_i}$ contains $F_{U_i}$.  The other possibility is that $P$ is a closed point of $X$.  Then the discrete valuation on $F_v$ restricts to a discrete valuation $v_P$ on $F_P$ that in turn restricts to $v$ on $F$; and the completion of $F_P$ at $v_P$ is just $F_v$.  Since $\alpha_P$ lies in $\Sha(F_P,G)$, the image of $\alpha_P$ in $H^1(F_v,G)$ is trivial.  But $\alpha_P$ is the image of $\alpha$ in $H^1(F_P,G)$.  Hence the
the image of $\alpha$ in $H^1(F_v,G)$ is also trivial.  Thus $\alpha$ lies in $\Sha(F,G)$.
\end{proof}

\begin{prop} \label{sha_exact_sequence}
Let $\wh X$ be a regular projective curve over a complete discrete valuation ring $T$, with function field F and closed fiber $X$.  Let $X_{(0)}$ be the set of closed points of $X$.  For any linear algebraic group $G$ over $F$, there is an exact sequence
\[1 \to \Sha_X(F,G) \xrightarrow{\iota} \Sha(F,G)  \xrightarrow{\phi}  \prod_{P \in X_{(0)}}{\!\!\!\!'} \  \Sha(F_P,G) \to 1 \eqno{(*)}\]
of pointed sets, in which the map $\iota$ is an inclusion.
\end{prop}

\begin{proof}
By Proposition~\ref{sha_inclusion}, there is an inclusion $\iota:\Sha_X(F,G) \to \Sha(F,G)$, given by viewing each as a subset of $H^1(F,G)$.  By Proposition~\ref{loc-to-glob_sha_map}, there is a surjection $\phi:\Sha(F,G) \to \prod'_{P \in X_{(0)}} \Sha(F_P,G)$, whose entries are the natural maps to $\Sha(F_P,G) \subseteq H^1(F_P,G)$.  The composition $\phi\iota$ is trivial, since
$\Sha_X(F,G)$ is the kernel of the map $H^1(F,G) \to \prod_{P \in X} H^1(F_P,G)$.  Also, any element of $\Sha(F,G)$ has trivial image in $H^1(F_P,G)$ for any codimension one point $P \in \wh X$, and in particular for the generic point of any irreducible component of $X$.  Hence any element in the kernel of $\phi$ is also in the kernel of $H^1(F,G) \to \prod_{P \in X} H^1(F_P,G)$; i.e.\ is in $\Sha_X(F,G)$.  So the sequence is exact.
\end{proof}

Proposition~\ref{sha_exact_sequence} shows that $\Sha_X(F,G)$ is the kernel of $\phi$.
But since $(*)$ is in general just an exact sequence of pointed sets, describing the kernel does not describe the fibers of $\phi$.  But as in Corollary~\ref{fibers_of_six-term_seq}, we can describe those fibers by using the bijection between $H^1(F, G)$ and $H^1(F, G^\tau)$ that takes the class $[\tau] \in H^1(F, G)$ of $\tau \in Z^1(F,G)$ to the neutral element.  Namely, the fiber of $\phi$ containing $[\tau] \in \Sha(F,G)$ is carried to the kernel of the corresponding map for the twisted group $G^\tau$; and this is just $\Sha_X(F,G^\tau)$, by Proposition~\ref{sha_exact_sequence} applied to $G^\tau$.  We thus obtain the following corollary:

\begin{cor} \label{sha_sequence_fibers}
In the situation of Proposition~\ref{sha_exact_sequence},
let $\tau \in Z^1(F,G)$ be a cocycle whose class $[\tau]$ lies in $\Sha(F,G)$.  Then the fiber of $\phi: \Sha(F,G) \to \prod'_{P \in X_{(0)}} \  \Sha(F_P,G)$ that contains $[\tau]$ is in natural bijection with $\Sha_X(F,G^\tau)$ as pointed sets.
\end{cor}

By Proposition~\ref{sha_exact_sequence},
if $\Sha(F_P,G)$ is trivial for every closed point $P$ on the closed fiber $X$ of some regular model $\wh X$ of $F$, then $\Sha(F,G) = \Sha_X(F,G)$.  A related result appears at Theorem~\ref{local-global_Sha_disconnected} below.  First we prove a lemma.

\begin{lemma} \label{local_sha_finite_group}
Let $E$ be the fraction field of a two-dimensional complete regular local ring $A$, and let $G$ be a finite constant group over~$E$.  Then $\Sha(E,G)$ is trivial.
\end{lemma}

\begin{proof}
Let $\alpha \in \Sha(E,G) \subseteq H^1(E,G)$.  Then $\alpha$ defines an \'etale $E$-algebra; let $B$ be the integral closure of $A$ in this algebra.
Thus the morphism $\Spec(B) \to \Spec(A)$ is finite and generically separable.
Moreover it is unramified over the codimension one points of $\Spec(A)$, since $\alpha \in \Sha(E,G)$.
Since $A$ is regular and $B$ is normal, it follows by Purity of Branch Locus that $\Spec(B) \to \Spec(A)$ is unramified, and hence is an \'etale cover.

Let $\{x,y\}$ be a system of parameters for $\Spec(A)$ at its closed point
$P$.  By hypothesis, the torsor defined by $\alpha$ becomes trivial
over the $x$-adic completion $E_x$ of $E$.  Thus there is a section of the \'etale cover $\Spec(B) \to \Spec(A)$ over
$\Spec(A_x)$, where $A_x$ is the completion of the local ring of~$A$ at the height one prime~$(x)$.  Let $V \subset \Spec(A)$ be the closed subset defined by $(x)$.
The \'etale cover $\Spec(B) \to \Spec(A)$ thus has a section over the
generic point $(x)$ of $V$, and hence has a section over all of $V$.  Hence
it has trivial residue field extension at a point lying over the closed point $P$ of $V$.
Since $\Spec(B) \to \Spec(A)$ is \'etale,
Hensel's Lemma implies that $\Spec(B) \to \Spec(A)$ has a section.
Thus the torsor defined by $\alpha$ has a section over $\Spec(E)$; i.e.\ $\alpha$ corresponds to the trivial $G$-torsor over $E$.
This shows that $\Sha(E,G)$ is trivial.
\end{proof}

Note that the above proof shows that Lemma~\ref{local_sha_finite_group} holds more generally if $G$ is instead assumed to be the generic fiber of a smooth finite group scheme over $A$.

If $G$ is a rational linear algebraic group over $F$, then $G(E) \to (G/G^0)(E)$ is surjective for every $E/F$.  Hence $H^1(E,G^0) \to H^1(E,G)$ is injective by the cohomology exact sequence
\cite[I, 5.5, Proposition~38]{Serre:CG}, and thus $\Sha(F_P,G^0) \subseteq \Sha(F_P,G)$ for every closed point $P \in X$.  For such groups, the vanishing of the a priori smaller set $\Sha(F_P,G^0)$ suffices to obtain the equality stated after Corollary~\ref{sha_sequence_fibers}, as the following result shows:

\begin{thm} \label{local-global_Sha_disconnected}
Let $\wh X$ be a regular projective curve over a complete discrete valuation ring $T$, with function field $F$ and closed fiber $X$.  Let $G$ be a linear algebraic group over $F$ with $G/G^0$ a finite constant group.  Assume
that for each $P\in X$ and each valuation $v$ on $F_P$, the homomorphism $G((F_P)_v)\rightarrow (G/G^0)((F_P)_v)$ is surjective.
Suppose moreover that $\Sha(F_P,G^0)$ is trivial for every closed point $P$ of $\wh X$.
Then $\Sha(F,G) = \Sha_X(F,G)$ as subsets of $H^1(F,G)$.
\end{thm}

\begin{proof}
Let $\wh X$ be a regular model of $F$ over $T$.  Let $X_{(0)}$ denote the set of closed points of its closed fiber $X$.
By Proposition~\ref{sha_exact_sequence}, it suffices
to show that $\Sha(F_P,G)$ is trivial for each closed point $P \in X_{(0)}$.

Let $P \in X_{(0)}$.  By \cite[I, 5.5, Proposition~38]{Serre:CG}, the sequence of pointed sets
\[H^1(F_P,G^0)  \to H^1(F_P,G) \to H^1(F_P,\bar G)\]
is exact, where $\bar G=G/G^0$.

Now let $\alpha \in \Sha(F_P,G) \subseteq H^1(F_P,G)$.  Then the image of $\alpha$ in
$H^1(F_P,\bar G)$ lies in $\Sha(F_P,\bar G)$ and hence is trivial by
Lemma~\ref{local_sha_finite_group}.
Thus $\alpha$ is the image of an element of $H^1(F_P,G^0)$. By the
surjectivity assumption this element necessarily lies in $\Sha(F_P,G^0)$ (see
the arguments in the first and last paragraphs of the proof of
Corollary~\ref{abstract_Sha_quotient}). 
But $\Sha(F_P,G^0)$ is trivial, and so
$\alpha$ is trivial.
This shows that $\Sha(F_P,G)$ is trivial and completes the proof.
\end{proof}

Results about the vanishing of $\Sha(E,G)$ for algebraic groups~$G$ over
fraction fields $E$ of regular complete local rings give applications of the
above theorem.  In particular, there is the next result, following a
suggestion of J.-L.~Colliot-Th\'el\`ene, and using a strategy that is similar
to that used in the proof of Lemma~\ref{local_sha_finite_group}.   Here, as in \cite{SGA3.3}, Exp.~XIX, D\'efinition~2.7, we say that a group scheme $G$ over a base scheme $S$ is \textit{reductive} if it is affine and smooth over~$S$, and its geometric fibers are connected and reductive (meaning they have trivial unipotent radical).

\begin{lemma} \label{local_sha_reductive_group}
Let $E$ be the fraction field of a two-dimensional complete regular local ring $A$, and let $G$ be a reductive group scheme over $A$. Then $\Sha(E,G)$ is trivial.
\end{lemma}

\begin{proof}
If $\alpha \in \Sha(E,G) \subseteq H^1(E,G)$, then the corresponding $G$-torsor over $\Spec(E)$ is unramified at the codimension one points of $\Spec(A)$.
It therefore follows that $\alpha$ is induced by an element $\alpha_A$ of $H^1(\Spec(A), G)$, as shown in \cite[Theorem~4.2(i)]{CPS}.
(That result considered not our space $\Spec(A)$ but rather a two-dimensional regular scheme that is projective over a complete discrete valuation ring; but the only hypothesis that was used in the proof was that the scheme is two-dimensional and regular with function field $E$.)

Let $\{x,y\}$ be a system of parameters for $\Spec(A)$ at its closed point $P$.
By hypothesis, the torsor defined by $\alpha_A$ becomes trivial over the $x$-adic completion $E_x$ of $E$.  According to \cite{Nis84}, since $G$ is reductive, the natural map $H^1(A_x,G) \to H^1(E_x,G)$ has trivial kernel, where the discrete valuation ring $A_x$ is the completion of the local ring of $A$ at the height one prime $(x)$.  Thus the image of $\alpha_A$ in $H^1(A_x,G)$ is trivial.  Hence so is its image in $H^1(A_x/xA_x,G)$.  Since $A/xA$ is a discrete valuation ring with fraction field $A_x/xA_x$, it follows from \cite{Nis84} that the image of $\alpha_A$ in $H^1(A/xA,G)$ is trivial.  Hence so is its image in $H^1(A/(x,y),G)$.  Thus the $G$-torsor over $A$ given by $\alpha_A$ has a rational point over the residue field $A/(x,y)$ at $P$.  But $A$ is complete.  So this point lifts to an $A$-point on the torsor.  Thus this torsor is trivial, and hence so is the $G$-torsor over $E$ defined by $\alpha$.  This shows that $\Sha(E,G)$ is trivial.
\end{proof}

\begin{remark}
The assertion cited from~\cite{Nis84} was stated for reductive groups, but the proof was given there only for semi-simple groups.  The proof in the general reductive case was given in \cite[Th\'eor\`eme~I.1.2.2]{Gille}.
\end{remark}

\begin{thm} \label{sha_conditions}
Let $F$ be a one-variable function field over the fraction field of a complete discrete valuation ring $T$, and let $\wh X$ be a regular model for $F$, with closed fiber $X$.  Let $G$ be a linear algebraic group over $F$ with $G/G^0$ constant, such that for each $P\in X $ and each valuation $v$ on $F_P$, the homomorphism $G((F_P)_v)\rightarrow (G/G^0)((F_P)_v)$ is surjective.
Then $\Sha_X(F,G)$ equals $\Sha(F,G)$ in each of the following situations:
\renewcommand{\theenumi}{\roman{enumi}}
\renewcommand{\labelenumi}{(\roman{enumi})}
\begin{enumerate}
\item \label{rational}
$G^0$ is a rational variety and the residue field of $T$ is algebraically closed of characteristic zero; or
\item \label{reductive}
$G^0$ is the generic fiber of a reductive group scheme over $\hat X$; or
\item \label{semi-simple}
$G^0$ is semi-simple and simply connected, and the residue field of $T$ is algebraically closed of characteristic zero.
\end{enumerate}
Hence $\Sha_X(F,G)$ is independent of the above choice of $\hat X$, provided the above hypotheses are satisfied.
\end{thm}

\begin{proof}
By Theorem~\ref{local-global_Sha_disconnected}, it suffices to show that
$\Sha(F_P,G^0)$ is trivial for every closed point $P$ of $\wh X$.
In case~(\ref{rational}) this follows from~\cite[Corollary~7.7]{BKG}; and in case~(\ref{reductive}) it follows from Lemma~\ref{local_sha_reductive_group} above, via base change from $\wh X$ to $\Spec(\wh R_P)$.  In case~(\ref{semi-simple}) it follows from the fact that $H^1(F_P,G^0)$, which contains $\Sha(F_P,G^0)$, is trivial by~\cite[Theorem~5.1]{COP}.
\end{proof}

In particular, if $G$ is rational
over $F$ and the residue field of $T$ is algebraically closed of characteristic zero, then
the above conclusions about $\Sha_X(F,G)$ hold. 
In this case, the last assertion of the theorem was previously observed (see the comment after Corollary~\ref{fs_free}).

Combining Theorem~\ref{sha_conditions} with Corollary~\ref{Sha_P-graph} then gives:

\begin{cor} \label{Sha_descrip}
Let $G$ be a rational linear algebraic group over $F$.  If the residue field
of $T$ is algebraically closed of characteristic zero, or if
condition~(\ref{reductive}) of Theorem~\ref{sha_conditions} holds, 
then $\Sha(F,G)$ is finite.  Under either of these hypotheses, $\Sha(F,G)$ is
trivial if and only if $G$ is connected or $F^\s=F$ (or, equivalently,
if the reduction graph of some regular model of $F$ is a tree). 
\end{cor}

It would be interesting to know if the conclusion of
Theorem~\ref{sha_conditions} always holds even without assuming any of the
three additional hypotheses.  If so, then the conclusion of
Corollary~\ref{Sha_descrip} would hold for rational linear algebraic groups
even without an assumption on the residue field of $T$.  
In particular, the obstruction $\Sha(F,G)$ would then vanish for $G$ rational and connected, just as $\Sha_X(F,G)$ does.  
Whether this is actually the case is a question that is currently being studied by a number of researchers.

\section{Applications}\label{sec_applications}

In this section we apply the previous results in order to obtain local-global
principles for quadratic forms and central simple algebras. We begin by
considering homogeneous spaces which are not necessarily torsors. 
\subsection{Applications to homogeneous spaces}\label{subsec_homogeneous}

Until now our focus has been on local-global principles for torsors, i.e.\ for
principal homogeneous spaces.  Some results also carry over to other
homogeneous spaces.  As in \cite{HHK}, if $H$ is an $F$-variety on which a
linear algebraic group acts, then we say that $G$ \textit{acts transitively on
  the points of} $H$  if $G(E)$ acts transitively on $H(E)$ for every field
extension $E$ of $F$.  As observed in Corollary~\ref{homogeneous}, if a
local-global principle holds for $G$-torsors then it holds for all homogeneous
$G$-spaces that satisfy the above transitivity assumption.  There, the context
was for factorization inverse systems.  But it holds in particular for the
situation of patches, i.e.\ that of Notation~\ref{notn_patches}, since by
Corollary~\ref{patching_equivalence} they form such a system satisfying the
hypotheses of Corollary~\ref{homogeneous}. 
Using Proposition~\ref{approx}, the corresponding property holds in the
context of points on the closed fiber, in parallel to \cite{HHK},
Theorem~3.7. 

\begin{thm} \label{homog_pts}
Let $F$ be a one variable function field over a complete discretely valued
field with valuation ring~$T$, and let $\wh X$ be a normal model for $F$ over $T$, with closed fiber $X$.
If $G$ is a linear algebraic group over $F$ such that $\Sha_X(F,G)$ is trivial, then there is a local-global principle with respect to points on the closed fiber $X$ for all $F$-varieties $H$ for which $G$ acts transitively on the points. That is, if $H(F_P)$ is non-empty for every $P \in X$, then $H(F)$ is non-empty. In particular, this local-global principle holds if $G$ is connected and rational.
\end{thm}

\begin{proof}
Let $X_1,\dots,X_r$ be the irreducible components of $X$, and let $\eta_i$ be
the generic point of $X_i$.  
Since $H(F_{\eta_i})$ is non-empty, so is $H(F_{U_i})$ for some affine dense
open subset $U_i$ of $X_i$ that does not meet any other $X_j$, 
by Proposition~\ref{approx}. 
Let $\mc U$ be the
collection of the sets $U_i$, and let $\mc P$ be the complement of their
union. As was pointed out at the beginning of Section~\ref{sec_splitcovers},
$\Sha_{\mc P}(F, G)$ is a subset of $\Sha_X(F, G)$ and thus trivial by assumption. Hence by Corollary~\ref{homogeneous} applied to the
factorization inverse system given by the patches, if $H(F_\xi) \ne
\varnothing$
for each $\xi \in \mc P \cup \mc U$, then $H(F) \ne \varnothing$. The last assertion now follows from Corollary~\ref{Sha_P-graph}.
\end{proof}

It would be desirable to prove an analogous assertion in the context of discrete valuations, viz.\ that if $\Sha(F,G)$ is trivial then for any $G$-space $H$ with transitive action, there is an $F$-point on $H$ provided that there is an $F_v$-point for each discrete valuation $v$.  In this general direction we have the following result:

\begin{thm} \label{homog_val}
Let $T$ be a complete discrete valuation ring with fraction field $K$, such
that its residue field $k$ is algebraically closed of characteristic zero.
Let $F$ be a one-variable function field over $K$ with a regular model $\wh X$ having closed fiber $X$; and let $G$ be a linear algebraic group over $F$ such that $\Sha(F,G)$ or $\Sha_{\hat X,X}(F,G)$ is trivial.  If $H$ is a smooth projective
$F$-variety with a $G$-action that is transitive on points, and if $H(F_v)$
is non-empty for every discrete valuation $v$ on $F$, then $H(F)$ is
non-empty.
\end{thm}

\begin{proof}
Since $\Sha_{\hat X,X}(F,G)$ is contained in $\Sha(F,G)$ by Proposition~\ref{sha_exact_sequence}, it suffices to prove the theorem in the case that 
$\Sha_{\hat X,X}(F,G)$ is trivial.

By Proposition~\ref{restriction}, for every closed point $P \in X$ and
discrete valuation $v$ on $F_P$, the restriction of $v$ to $F$ is a discrete
valuation $v_0$.  Hence for every such $P$, $H((F_P)_v)$ is non-empty, as it
contains $H(F_{v_0})$.  Since $H$ is smooth, Corollary~5.7 of \cite{CGP} applies, and thus
$H(F_P)$ is non-empty.

Next, consider a generic point $\eta$ of an
irreducible component $X_0$ of $X$.  Then $H(F_\eta)$ is non-empty since
$F_\eta = F_v$, where $v$ is the discrete valuation on $F$ corresponding to
the codimension one point $\eta \in \wh X$.

Thus $H(F_P)$ is non-empty for every point $P \in X$.  Since $\Sha_X(F,G)$ is trivial, it follows from Theorem~\ref{homog_pts} that $H(F)$ is non-empty.
\end{proof}

Note that the containment $\Sha_{\hat X,X}(F,G) \subseteq \Sha(F,G)$ also shows that 
Theorem~\ref{homog_pts} remains true if instead $\Sha(F,G)$ is assumed trivial.  Also, 
as in Theorem~\ref{homog_pts}, the triviality hypothesis in Theorem~\ref{homog_val} in particular holds 
for $G$ rational and connected, by Corollary~\ref{Sha_P-graph}.

\subsection{Applications to quadratic forms}\label{subsec_quadforms}

Below we prove local-global principles for quadratic forms in terms of points on the closed fiber and in terms of valuations.  These results concern whether a quadratic form is isotropic or hyperbolic; the value of the Witt index of a form; and the Witt group of a field.

As before, let $T$ be a complete discrete valuation ring with uniformizer $t$, fraction field $K$, and residue field $k$, and let $F$ be a one-variable function field over $K$.  In this section, we assume that $K$ does not have characteristic $2$. Let $\wh X $ be a normal model of $F$ over $T$.

In the situation of Notation~\ref{notn_patches}, a local-global principle for isotropy was shown in \cite[Theorem~4.2]{HHK} for quadratic forms over $F$ of dimension unequal to two, in terms of the fields $F_P$ and $F_U$.
An analogous local-global principle in terms of completions $F_v$ at
valuations was shown in~\cite[Theorem~3.1]{CPS}, using the result
in~\cite{HHK} and a theorem of Springer (\cite[Proposition~VI.1.9(2)]{Lam});
this is an analog of the classical Hasse-Minkowski theorem.  (The statement of ~\cite[Theorem~3.1]{CPS} assumed that some model has a smooth generic fiber, but this was not essential.)
The next result proves the analog in terms of points on the closed fiber $X$ of $\wh X$.  (We have learned that 
in the case that $\cha(k) \ne 2$ and $\wh X$ is regular, 
this result also appears in the Ph.D. thesis of David Grimm.  This case also follows from \cite[Theorem~3.1]{CPS} together with Proposition~\ref{field_containment}.)

\begin{thm} \label{isotropy_pts_dim_not_2}
Let $F$ and $\wh X$ be as above.
If $q$ is a quadratic form over $F$ of dimension unequal to two, and if $q$ is isotropic over $F_P$ for every $P \in X$, then $q$ is isotropic over $F$.
\end{thm}

\begin{proof}
Without loss of generality, we may assume that $q$ is regular of dimension $n\geq 3$.
Then $\Orth(q)$ and its identity component $\operatorname{SO}(q)$ act transitively on the points of the projective quadric hypersurface $H$ defined by $q$ (for details, see the proof of Theorem~4.2 of~\cite{HHK}).
By Remark~4.1 of that article, the group $\operatorname{SO}(q)$ is rational. Thus by Theorem~\ref{homog_pts}, $H(F)$ is non-empty provided that each $H(F_P)$ is; this is equivalent to the desired assertion.
\end{proof}

\begin{example} \label{octonion example}
\renewcommand{\theenumi}{\alph{enumi}}
\renewcommand{\labelenumi}{(\alph{enumi})}
\begin{enumerate}
\item \label{octonions_points}
In the situation of Theorem~\ref{isotropy_pts_dim_not_2}, there are local-global principles for octonion algebras and for torsors under linear algebraic groups of type $G_2$.  
Namely, these algebras and these torsors are each
in bijection with the isomorphism classes of three-fold Pfister forms 
$\<1,a\>\otimes\<1,b\>\otimes\<1,c\>$ (\cite{Se:Bo}, Section~8.1, Th\'eor\`eme~9).  Since such a quadratic form has dimension eight, 
Theorem~\ref{isotropy_pts_dim_not_2} applies.  
But a Pfister form that is isotropic is in fact hyperbolic \cite[Theorem~X.1.7]{Lam}.  
Thus an octonion algebra or a group of type $G_2$ is split over $F$ if and only if it becomes split over each $F_P$, for $P \in X$.  
(This can also be seen by using that $F$-groups of type $G_2$ are classified by 
$H^1(F,\Aut(G))=H^1(F,G)$ for $G$ any group of type $G_2$; and that if $G$ is split then it is rational and hence Corollary~\ref{Sha_P-graph} applies.)
\item \label{octonions_valuations}
If the characteristic of the residue field $k$ is also assumed to be unequal to two, then the assertions in part~(\ref{octonions_points}) of this remark also apply to the corresponding local-global principles with respect to discrete valuations, by citing~\cite[Theorem~3.1]{CPS} instead of Theorem~\ref{isotropy_pts_dim_not_2}.
\end{enumerate}
\end{example}

Recall that two quadratic forms $q,q'$ are \textit{Witt equivalent} if $q\perp
h\cong q' \perp h'$, where $h,h'$ are hyperbolic forms.  The Witt equivalence classes form the {\em Witt group} $W(F)$, which is also a ring under multiplication given by tensor product.
By Witt
decomposition (\cite[Theorem~I.4.1]{Lam}), every regular form $q$ is Witt equivalent to an anisotropic form; i.e.\ $q = q_a \perp q_h$ with $q_a$ anisotropic and $q_h$ hyperbolic.  Here the
\textit{Witt index} of $q$ is $i_W(q) = \frac{1}{2}\dim\,q_h$.

\begin{cor} \label{local_witt_index_not_2}
If $q$ is a regular quadratic form then
\[i_W(q) = \min\limits_{P \in X} i_W(q_{F_P})
         \eqno{(*)}\]
unless $q$ is Witt equivalent to an anisotropic binary form that becomes isotropic over each $F_P$.  In that exceptional case,
         \[i_W(q) = \min\limits_{P \in X} i_W(q_{F_P})-1.
\eqno{(**)}\]
\end{cor}

\begin{proof}
Write $q=q_a \perp q_h$, as above.  If $q_a$ is not binary, then since it is anisotropic over $F$, it is also anisotropic over some $F_P$, by Theorem~\ref{isotropy_pts_dim_not_2}.
The equality $(*)$ now follows.

If $q_a$ is binary and some $(q_a)_P$ is anisotropic, then $i_W(q_a) = \min_{P
  \in X} i_W((q_a)_{F_P}) = 0$.  Thus $(*)$ holds in this case.  
In the remaining (exceptional) case, $i_W(q_a) =0$ and $\min_{P \in X} i_W((q_a)_{F_P}) = 1$, hence $(**)$ holds.
\end{proof}

For any rational linear algebraic group $G$ over $F$, the pointed set $\Sha_0(F,G)$ is equal to $\Sha_{\wh X,X}(F,G)$, for any regular model $\wh X$ of $F$ over $T$, with closed fiber $X$. 
Recall that under Hypothesis~\ref{notn_branched}, the fundamental group of the reduction graph of a regular model $\wh X$ also depends only on $F$, by
Corollary~\ref{fs_free}. As a result, we may simply write $\pi_1(\Gamma)$ for this common fundamental group, where
$\Gamma$ is the reduction graph of any regular model of $F$. 

\begin{thm} \label{Witt_loc-gl}
Let $F$ be a one variable function field over a complete discretely valued field $K$ of characteristic not equal to~$2$, and let $\wh X$ be a regular model of $F$.
Then the kernel of the natural homomorphism of Witt groups
$\pi:W(F) \to \prod\limits_{P \in X} W(F_P)$
is isomorphic to $\Sha_0(F,\mbb Z/2\mbb Z)$. Moreover, both groups are isomorphic to the elementary abelian two-group $\Hom(\pi_1(\Gamma),\mbb Z/2 \mbb Z)$, and each element in the kernel is represented by a quadratic form of dimension two.
\end{thm}

\begin{proof}
Let $h$ denote a hyperbolic plane.  Then $\SOrth(h)$ is rational by \cite{HHK}, Remark~4.1; and hence so is $\Orth(h)$, since each of the two components has a rational point.  Now 
$H^1(F,\Orth(h))$ classifies the equivalence classes of regular two-dimensional quadratic forms over $F$, with the distinguished element corresponding to the quadratic form $h$ (see \cite[Proposition~VII.29.1 and VII.29.28]{BofInv}).
Each element in $\ker \pi \subseteq W(F)$ is represented by an anisotropic quadratic form that becomes hyperbolic over each $F_P$.  Since a non-trivial hyperbolic form is isotropic, it follows from
Theorem~\ref{isotropy_pts_dim_not_2} that such a form is binary.
Thus there is a natural bijection of pointed sets $\ker \pi \to
\Sha_0(F,\Orth(h))$. Since $\Orth(h)$ is rational, and its quotient by its
identity component is isomorphic to $\mbb Z/2 \mbb Z$, there is a bijection of
$\Sha_0(F,\Orth(h))$ to $\Sha_0(F,\mbb Z/2 \mbb Z)$, by
Corollary~\ref{Sha_P-graph}. We claim that the composition $\ker \pi\to
\Sha_0(F,\mbb Z/2 \mbb Z)$ is a homomorphism, and hence an isomorphism. It
suffices to show that the composition $\ker \pi\to \Sha_0(F,\mbb Z/2 \mbb Z)\subseteq
H^1(F,\mbb Z/2 \mbb Z) = F^\times/(F^\times)^2$ is.  But this composition
takes the diagonal quadratic form $\< 1, -a \>$ to the square class of $a$,
for $a \in F^\times$. Since $\<1,-a\>\perp\<1,-b\>$ is Witt equivalent to
$\<1,-ab\>$, the claim follows. 

The remaining assertion now follows directly from Corollary~\ref{Sha_P-finite_gp}.
\end{proof}

Thus, the local-global map $\pi:W(F) \to \prod\limits_{P \in X} W(F_P)$ on
Witt groups has trivial kernel if and only if $\pi_1(\Gamma)=1$.  More
generally, if $\pi_1(\Gamma)$ is free of rank~$r$, then $\ker(\pi) = (\mbb Z/2
\mbb Z)^r$. 

\begin{cor} \label{tree_implications}
In the above situation, the following are equivalent:
\renewcommand{\theenumi}{\roman{enumi}}
\renewcommand{\labelenumi}{(\roman{enumi})}
\begin{enumerate}
\item \label{all_isotropy}
The local-global principle for isotropy in Theorem~\ref{isotropy_pts_dim_not_2} holds for \textbf{all} binary quadratic forms over $F$.
\item \label{all_index_equality}
The equality $(*)$ of Corollary~\ref{local_witt_index_not_2} holds for \textbf{all} quadratic forms over $F$.
\item \label{tree_condition}
$F=F^{\split}$ or equivalently, the reduction graph of a regular model of $F$ is a tree.
\end{enumerate}
\end{cor}

\begin{proof}
A form of dimension two is isotropic if and only if it is hyperbolic.
So the equivalence of (\ref{all_isotropy}) and (\ref{tree_condition})
follows from Theorem~\ref{Witt_loc-gl}, since the reduction graph $\Gamma$ is a tree if and only if $\pi_1(\Gamma)$ is trivial.

By Corollary~\ref{local_witt_index_not_2}, the equality $(*)$ is equivalent to the vanishing of the kernel of the local-global map on Witt groups.  So the equivalence of (\ref{all_index_equality}) and (\ref{tree_condition}) follows from Theorem~\ref{Witt_loc-gl} and Corollary~\ref{fs_free}.
\end{proof}

\begin{remark}
As with Theorem~\ref{isotropy_pts_dim_not_2}, the other results above also
have analogs in terms of patches on a regular model; the analog for
Corollary~\ref{local_witt_index_not_2} is in fact Corollary~4.3 of~\cite{HHK}.
\end{remark}

We next turn to analogs of the above three results for the set of discrete valuations on $F$, thereby extending the results of \cite{CPS}.
We first prove preliminary results, using the theorem of Springer cited above.  For this reason we assume that the residue field $k$ of the discrete valuation ring has characteristic unequal to two, as in
\cite{CPS}.

\begin{lemma} \label{val_local}
Let $R$ be a regular complete local domain of dimension two, whose residue field $k$ has characteristic unequal to two.  Let $E$ be the fraction field of $R$; let $\{x,y\}$ be a generating set for the maximal ideal of $R$; and let $E_y$ be the completion of $E$ with respect to the $y$-adic valuation.
Let $q = \sum_{i=1}^n a_iZ_i^2$ be a diagonal quadratic form over $R$ such that each $a_i$ has the form $x^{r_i} y^{s_i} u_i$ for some $r_i,s_i \ge 0$ and some unit $u_i \in R^\times$.
\begin{enumerate}
\item \label{local_isotropy}
If $q$ is isotropic over $E_y$ then it is isotropic over $E$.
\item \label{local_hyperbolic}
If $q$ is hyperbolic over $E_y$ then it is hyperbolic over $E$.
\end{enumerate}
\end{lemma}

\begin{proof}
(\ref{local_isotropy}) We follow the strategy used in the proof of~\cite[Theorem~3.1]{CPS}.  After factoring out even powers of $x$ and $y$,
we may assume that the exponents of $x$ and $y$ in the coefficients $a_i$ of $q$ are each equal to $0$ or $1$.
We may now write
$q = q_1 \perp xq_2 \perp yq_3 \perp xyq_4$, where each $q_i$ is a regular quadratic form over $R$.
Since $q = (q_1 \perp xq_2) \perp y(q_3 \perp xq_4)$ is isotropic over the complete discretely valued field $E_y$, it follows from a theorem of Springer (\cite[Proposition~VI.1.9(2)]{Lam}) that $q_1 \perp xq_2$ or $q_3 \perp xq_4$ is isotropic over
the residue field of $E_y$, i.e.\ over
the fraction field of the discrete valuation ring $\bar R := R/yR$.
Applying Springer's theorem to that subform over $\bar R$ yields that one of the forms $q_i$ is isotropic over the residue field $k$ of $\bar R$, and hence so is $q$.  Thus there is a non-zero $k$-point on the affine quadric hypersurface $Q \subset \mbb A_R^n$ defined by $q$.
By the assumptions on $q$ and on $k$,
this point of $Q$ is smooth because it is not the origin. Since $k$ is also the residue field of the complete ring $R$, \cite[Lemma~4.5]{HHK} implies that this $k$-point lifts to an $R$-point of $Q$.  Hence $q$ is isotropic over $E$.

(\ref{local_hyperbolic}) By Witt decomposition, we may write $q = q' \perp q''$, where $q'$ is anisotropic over $E$ and $q''$ is hyperbolic over $E$.  Since $q$ and $q''$ are hyperbolic over $E_y$, so is $q'$, by Witt cancellation (\cite[Theorem~I.4.2]{Lam}).  If $q'$ is not the zero form, then it is isotropic over $E_y$, and hence it is isotropic over $E$ by part~(\ref{local_isotropy}).  This contradiction shows that $q'=0$ and hence $q=q''$ is hyperbolic.
\end{proof}

In this proposition we return to our standing notation of this section.

\begin{prop} \label{good_model}
Assume that the residue field $k$ of $T$ has characteristic unequal to two. Let $F$ be a one-variable function field over the fraction field of $T$ and let $q$ be a quadratic form over $F$. 
\renewcommand{\theenumi}{\alph{enumi}}
\renewcommand{\labelenumi}{(\alph{enumi})}
\begin{enumerate}
\item \label{good model exists}
Then there is a regular model $\wh X$ of $F$ such that for every point $P$ on the closed fiber of $\wh X$, $q$ is isotropic (resp.\ hyperbolic) over $F_P$ if and only if it is isotropic (resp.\ hyperbolic) over $(F_P)_v$ for every discrete valuation $v$ on $F_P$. 
\item \label{good model property} 
Moreover, for a model $\wh X$ as in part~(\ref{good model exists}), the given form 
$q$ is isotropic (resp.\ hyperbolic) over $F_P$ for
every point $P$ on the closed fiber of $\wh X$
if and only if it is isotropic (resp.\ hyperbolic) over 
$F_v$ for every discrete valuation $v$ on $F$.
\end{enumerate}
\end{prop}

\begin{proof}
(\ref{good model exists}) The forward implication is true for any regular model, so to prove that assertion it suffices to find a model for which the reverse implication holds. Let $\wh X'$ be any regular model for $F$.
Since $\cha(F) \ne 2$, after replacing $q$ by an equivalent form we may assume that $q$ is a diagonal form $\sum a_i Z_i^2$, with $a_i \in F^\times$.
Let $D'$ be the union of the supports of the divisors of the elements $a_i$ as rational functions on $\wh X'$.  For some blow-up $f:\wh X \to \wh X'$, the singularities of $D := f^{-1}(D')$ are normal crossings (e.g.\ see~\cite[Lemma~4.7]{HHK}).
Let $X$ be the closed fiber of $\wh X$, and let $P\in X$. If $P$ is the generic point of a component of $X$, then $F_P$ is a complete discretely valued field, and the statement is trivial. So let $P$ be a closed point.  By the above condition on the divisor $D$, we may choose local parameters $x,y\in \wh R_P$ such that the components of $D$ that meet $P$ lie in the zero locus of $xy$ on $\wh X$.
After multiplying $q$ by an element of the form $x^r y^s$, we may assume that $q$ is as in the hypothesis of Lemma~\ref{val_local}. If $q$ is isotropic (resp.\ hyperbolic) over each completion of $F_P$, it is in particular isotropic (resp.\ hyperbolic) over the completion of $F_P$ with respect to the $y$-adic valuation on $F_P$. Then the lemma implies that $q$ is isotropic (resp.\ hyperbolic) over $F_P$, as we needed to show.

(\ref{good model property}) The forward implication 
follows from Proposition~\ref{field_containment}.  For the reverse implication, suppose that $q$
is isotropic (resp.\ hyperbolic) over every $F_v$.
Every discrete valuation on $F_P$ restricts to a discrete valuation on $F$ (Proposition~\ref{restriction}), and so $q$ is isotropic (resp.\ hyperbolic) over all completions of $F_P$ with respect to its discrete valuations.  Hence by part~(\ref{good model exists}), $q$ is isotropic (resp.\ hyperbolic) over each $F_P$.
\end{proof}

We now can now state analogs of Corollary~\ref{local_witt_index_not_2},
Theorem~\ref{Witt_loc-gl}, and Corollary~\ref{tree_implications}, in the context of discrete valuations.
(As noted above, the corresponding analog of Theorem~\ref{isotropy_pts_dim_not_2} was proven in~\cite[Theorem~3.1]{CPS}.)

\begin{thm} \label{val_global}
Assume that the residue field $k$ of the complete discrete valuation ring $T$ has characteristic unequal to two.  Let $F$ be a one-variable function field over the field of fractions $K$ of $T$, and let $\Gamma$ be the reduction graph of a regular model.
\begin{enumerate}
\item \label{val_isotropy}
The quadratic forms on $F$ satisfy a local-global principle for isotropy with respect to the set of discrete valuations $\Omega_F$ (i.e.\ an arbitrary form $q$ is isotropic over $F$ if it is isotropic over $F_v$ for each $v\in \Omega_F$) if and only if $\Gamma$ is a tree.
\item \label{val_Witt_index}
For any regular quadratic form $q$ over $F$,
its Witt index $i_W(q)$ is equal to either
$\operatorname{min}_{v \in \Omega_F} i_W(q_{F_v})$ or
$\operatorname{min}_{v \in \Omega_F} i_W(q_{F_v})-1$.  The second case
occurs precisely for those forms that are Witt equivalent to an anisotropic binary form that becomes isotropic over each $F_v$.  Moreover the first case holds for {\bf all} regular quadratic forms $q$ over $F$ if and only if $\Gamma$ is a tree.
\item \label{val_Witt}
The kernel of the natural homomorphism of Witt groups
$\varpi:W(F) \to \prod\limits_{v \in \Omega_F} W(F_v)$ is
equal to the kernel of $\pi:W(F) \to \prod\limits_{P \in X} W(F_P)$ for any regular model of $F$.  Thus it is
isomorphic to the elementary abelian two-group $\Hom(\pi_1(\Gamma),\mbb Z/2 \mbb Z)$, and each element in the kernel is represented by a quadratic form of dimension two.
\renewcommand{\labelenumi}{(\alph{enumi})}
\end{enumerate}
\end{thm}

\begin{proof}
First observe that by Corollary~\ref{fs_free}, the choice of a regular model does not affect 
the isomorphism class of the reduction graph, and in particular does not affect
whether the reduction graph is a tree.  Thus we may pick a regular model of our choosing in 
parts~(\ref{val_isotropy}) and~(\ref{val_Witt_index}).  
This observation will also be useful in the proof of part~(\ref{val_Witt}).

(\ref{val_isotropy})
If the quadratic forms on $F$ satisfy a local-global principle with respect to discrete valuations, then they also satisfy such a principle with respect to the points $P$ on the closed fiber of any regular model, since every $F_v$ contains an $F_P$ by Corollary~\ref{field_containment}. Hence $\Gamma$ is a tree, by Corollary~\ref{tree_implications}.

For the converse direction, suppose that $\Gamma$ is a tree. Consider a quadratic form $q$ over $F$ that becomes isotropic over each $F_v$. We need to show it is isotropic over $F$. Let $\wh X$ be a model given by Proposition~\ref{good_model}(\ref{good model exists}).
By part~(\ref{good model property}) of that proposition,
$q$ is isotropic over $F_P$ for
every point $P$ on the closed fiber of $\wh X$.
But then $q$ is isotropic over $F$,
by Corollary~\ref{tree_implications}.

(\ref{val_Witt_index})
Given a regular quadratic form $q$ over $F$, let $\wh X$ be as in Proposition~\ref{good_model}(\ref{good model exists}), with closed fiber $X$. 
To prove the first assertion in part~(\ref{val_Witt_index}), 
it suffices by Corollary~\ref{local_witt_index_not_2} to show that
$\min\limits_{v \in \Omega} i_W(q_{F_v}) = \min\limits_{P \in X} i_W(q_{F_P})$.  The former expression is greater than or equal to the latter, since every $F_v$ contains an $F_P$.
By Proposition~\ref{good_model}(\ref{good model property}), those quantities
are equal if $q$ has dimension two.  It remains to show that the above inequality is actually an equality if $\dim(q) > 2$.  In this case the right hand side is equal to $i_W(q)$, by Corollary~\ref{local_witt_index_not_2}.
Write $q = q_a \perp q_h$ with $q_a$ anisotropic and $q_h$ hyperbolic.
Since $i_W(q) = \min\limits_{P \in X} i_W(q_P)$, the form $q_a$ is anisotropic
over $F_P$ for some point $P \in X$.  Given $P$, the form
$q_a$ remains anisotropic over the completion of $F_P$ with respect to some discrete valuation on $F_P$, by the defining property of $\wh X$ (as in Proposition~\ref{good_model}(\ref{good model exists})).  By Proposition~\ref{restriction} it also remains anisotropic over $F_v$ for the restriction $v$ of that valuation to $F$. Hence $i_W(q_a)=i_W((q_a)_{F_P})=i_W((q_a)_{F_v})=0$ and thus $i_W(q)=i_W(q_{F_v})$, as needed.

The second assertion in part~(\ref{val_Witt_index}) now follows from 
Corollary~\ref{local_witt_index_not_2} and Proposition~\ref{good_model}(\ref{good model property}); 
and the third assertion in~(\ref{val_Witt_index}) then follows from the assertion of part~(\ref{val_isotropy}) above. 

(\ref{val_Witt})
Given any  regular model $\wh X$ of $F$ with closed fiber $X$,  since every $F_v$ contains an $F_P$ by Corollary~\ref{field_containment}, it follows that $\ker(\pi)$ is contained in $\ker(\varpi)$.  

For the reverse containment, let $q$ be a regular quadratic form $q$ whose class lies in $\ker(\varpi)$.  First consider the case that $\wh X$ is a model as given in Proposition~\ref{good_model}(\ref{good model exists}), 
relative to $q$.  
Again using Proposition~\ref{restriction}, the class of $q$ also becomes trivial over each completion of a field $F_P$ (with respect to some discrete valuation), for every $P$ in the closed fiber $X$ of $\wh X$. Proposition~\ref{good_model}(\ref{good model exists}) then implies that the class of $q$ lies in
$\ker(\pi)$; implying $\ker(\varpi) \subseteq \ker(\pi)$ and hence equality. 

More generally, for an arbitrary regular model $\wh X'$, the reduction graphs of $\wh X$ and $\wh X'$ have isomorphic fundamental groups, as observed above.
Hence by Theorem~\ref{Witt_loc-gl}, $\ker(\pi)$ and the corresponding kernel $\ker(\pi')$ for $\wh X'$ are isomorphic and finite. But since $\ker(\pi')\subseteq\ker(\varpi)=\ker(\pi)$, this implies $\ker(\pi')=\ker(\varpi)$, proving the desired equality.

The remaining assertions follow from Theorem~\ref{Witt_loc-gl}.
\end{proof}

The above result also shows that the local-global maps $\varpi$ and $\pi$ in part~(\ref{val_Witt}) have kernel contained in the fundamental ideal $I(F) \subset W(F)$, and that the kernel meets $I^2(F)$ trivially.  Hence the restrictions of $\varpi$ and $\pi$ to $I^2(F)$ are injective, yielding a local-global principle for such classes of forms.  Compare the assertions of \cite[Theorem~3.10]{COP} and \cite[Lemma~4.10]{Hu}, in a more local situation.

\subsection{Applications to central simple algebras}\label{subsec_csa}

We next consider applications of our results to central simple algebras.  We first carry over a local-global assertion about the index of an algebra from the context of patches to the contexts of points on the closed fiber and of valuations.  Afterwards we prove a local-global result about central simple algebras with unitary involutions.

Recall that in \cite[Theorem~5.1]{HHK} we showed that under Notation~\ref{notn_patches},
the index of a central simple $F$-algebra $A$ is the least common multiple of the indices of the central simple $F_\xi$-algebras $A_\xi := A \otimes_F F_\xi$, for $\xi \in \mc P \cup \mc U$.  Here, we obtain an analogous local-global principle for points on the closed fiber:

\begin{thm} \label{lgp_index}
Under Notation~\ref{notn_fields}, let $A$ be a central simple $F$-algebra.  Then
$$\ind(A) = \mathop{\mathrm{lcm}}\limits_{P \in X} \ind(A_P),\eqno{(*)}$$
where $A_P := A \otimes_F F_P$ is viewed as a central simple $F_P$-algebra, for $P \in X$.  In particular, $A$ is split over $F$ if and only if each $A_P$ is split over $F_P$.
\end{thm}

\begin{proof}
Let $\SB_i(A)$ be the generalized Severi-Brauer variety associated to $A$.  By 
\cite[Proposition~1.17]{BofInv}, 
if $E/F$ is a field extension, then
$\SB_i(A)(E) \neq \varnothing$ if
and only if $\ind(A_E)$ divides~$i$.  Moreover, the rational connected linear algebraic group $\GL_1(A)$ acts transitively on the points of $\SB_i(A)$.  So Theorem~\ref{homog_pts} implies that $\ind(A)\mid i$ if and only if $\ind(A_P)\mid i$ for each $P\in X$; this is equivalent to the desired assertion.
\end{proof}

We can also consider an analog of Theorem \ref{lgp_index} in terms of
discrete valuations on $F$, parallel to the result in the case of quadratic
forms that was given in \cite[Theorem~3.1]{CPS}.  The following result, which is a valuation version of the last assertion of the above theorem, also appeared in \cite[Theorem~4.3(ii)]{CPS}, and can be viewed as an analog of the Albert-Brauer-Hasse-Noether theorem:

\begin{cor} \label{lgp_csa_split_val}
Let $F$ be a one-variable function field over a complete discretely valued field, and let $A$ be a central simple $F$-algebra.  Then $A$ splits over $F$ if and only if it splits over $F_v$ for every discrete valuation $v$ on $F$.  Moreover two such algebras $A,A'$ are isomorphic over $F$ if and only if they become isomorphic over each $F_v$.
\end{cor}

\begin{proof}
Let $\wh X$ be a regular model for $F$ over the valuation ring $T$, with closed fiber $X$.  The isomorphism classes of central simple $F$-algebras of degree $d$ are classified by $H^1(F,\PGL_d)$ (\cite{BofInv}, p.~396), with the isomorphism class of a split algebra corresponding to the trivial cohomology class.  Since the connected linear algebraic group $\PGL_d$ is defined and reductive over $\wh X$, Theorem~\ref{sha_conditions}(\ref{reductive}) implies that $\Sha(F,\PGL_d) = \Sha_X(F,\PGL_d)$.  But $\Sha_X(F,\PGL_d)$ is trivial by Theorem~\ref{lgp_index}.  Hence so is $\Sha(F,\PGL_d)$, and this implies the first assertion.
The second assertion now follows from the first by considering the difference of the two Brauer classes, and using that $A,A'$ have the same degree.
\end{proof}

\begin{remark}  
\begin{enumerate}
\item \label{Saltman pf rk}
We sketch another proof of the above corollary, using ideas of \cite{Sal:DA}, Section~3 (see also \cite{Sal:DAC}). Namely, letting $X_1,\dots,X_n$ be the irreducible components of the closed fiber $X$ of a regular model $\wh X$ of $F$, with generic points $\eta_i$ and residue fields $\kappa_i$, the composition
$\Br(\wh X) = \Br(X) \hookrightarrow \prod_i \Br(X_i) \hookrightarrow \prod_i
  \Br(\kappa_i)$ also factors as $\Br(\wh X) \to \prod \Br(\wh R_{\eta_i}) \to \prod_i \Br(\kappa_i)$.  
But any element $\alpha \in \Br(F)$ that splits over each $F_v$ must lie in $\Br(\wh X)$; and its image in $\prod_i \Br(\kappa_i)$ is trivial
because it splits over each $\wh R_{\eta_i}$ 
(using that we have an inclusion $\Br(\wh R_{\eta_i}) \hookrightarrow \Br(F_{\eta_i}) = \Br(F_{v_{\eta_i}})$).
Since the first composition above is an inclusion, $\alpha$ is trivial, as asserted.  The proof of the corollary above instead relies on  \cite{Nis84}
(by using Theorem~\ref{sha_conditions}(\ref{reductive})
and hence Lemma~\ref{local_sha_reductive_group}), which can be viewed as  generalizing this approach to more general reductive groups.
\item \label{Suresh-Reddy rk}
In the presence of sufficiently many roots of unity, a valuation analog of the first assertion in Theorem~\ref{lgp_index} can also be shown, again by drawing on \cite[Theorem~5.1]{HHK}.  See Theorem~2.6 of \cite{RS}.
\end{enumerate}
\end{remark}

Turning to the second application, we consider central simple $F$-algebras with involutions, where $F$ is as before.
Recall that an involution on an $F$-algebra $A$ is {\em unitary} if its restriction to the center of $A$ is an automorphism of order two.
Given a quadratic \'etale $F$-algebra $E$, a {\em central simple $F$-algebra with unitary involution} is a finite dimensional $F$-algebra $A$ with unitary involution such that $F$ is the subfield of the center of $A$ that consists of the elements fixed by the involution.  (See~\cite[pp.~20-21 and p.~400]{BofInv}.)

\begin{thm} \label{csa_involution}
Under Notation~\ref{notn_patches}, let $E/F$ be a quadratic \'etale algebra such that $E \otimes_F F_\wp
\cong F_\wp \times F_\wp$ for every $\wp \in \mc B$.  Let $(B, \tau),
(B', \tau')$ be two central simple $F$-algebras with unitary involutions
and centers isomorphic to $E$.
Then $(B, \tau)$ and $(B', \tau')$ are isomorphic as
algebras with involutions provided that they are locally isomorphic, in either of the following senses:

\begin{enumerate}
\item \label{unitary_patches}
They become isomorphic over each field $F_P$ for $P \in \mc P$, and for each field $F_U$ for $U \in \mc U$.

\item \label{unitary_points}
They become isomorphic over each field $F_P$, for $P$ a point on the closed fiber of $\wh X$.
\end{enumerate}
\end{thm}

\begin{proof}
(\ref{unitary_patches})
Let $\Aut(B, \tau)$ denote the automorphism group of $(B,\tau)$ over $F$, and let $\Aut_E(B, \tau)$ denote the automorphism group over $E$.  Then $\Aut_E(B, \tau)$ can be identified with $\PGU(B, \tau)$;
algebras with unitary involution and the same degree as $B$ are
classified by the cohomology set $H^1(F, \Aut(B, \tau))$; and there is an exact sequence
\[1 \to \PGU(B, \tau) \to \Aut(B, \tau) \to S_2 \to 1 \]
(see \cite[page
400]{BofInv}).
Here the map on the right is obtained by restricting the automorphism
to $E$, and the symmetric group $S_2$ is identified with the group of
automorphisms of $E/F$. In particular, the trivial element of $H^1(F, S_2)$
corresponds to the isomorphism class of $E$.

Suppose that $(B, \tau)$
and $(B', \tau')$ are as above and become isomorphic over each $F_\xi$ for $\xi \in \mc P \cup \mc U$.
Hence $(B', \tau')$ induces the trivial class in
$H^1(F_\xi, \Aut(B, \tau))$ for each $\xi \in \mc P \cup \mc U$, and
hence it lies in $\Sha_{\mc P}(F, \Aut(B, \tau))$.  Also, its image in $H^1(F, S_2)$ is trivial.  So by Corollaries~\ref{abstract_Sha_quotient}
and~\ref{patching_equivalence},
this class lies in the image of $\Sha_{\mc P}(F,\PGU(B, \tau))$.
To prove the assertion it thus suffices to show that $\Sha_{\mc P}(F,\PGU(B, \tau))$ is trivial.

To do this, we first observe that there is an exact sequence
\[1 \to R^1_{E_\wp/F_\wp} \mbb G_m \to \oper{U}(B, \tau)_{F_\wp} \to \PGU(B,
\tau)_{F_\wp} \to 1,\eqno{(*)}\]
where $R^1_{E_\wp/F_\wp} \mbb G_m$ denotes the elements 
of $R_{E_\wp/F_\wp} \mbb G_m$ of $\tau$-norm $1$.
This exact sequence is constructed as follows.  As in \cite[p.~401]{BofInv}, there is an exact sequence
\[1 \to R_{E/F} \mbb G_m \to \GU(B, \sigma) \to \PGU(B, \sigma) \to 1.\]
If we consider the $\tau$-norm $1$ subgroup of $\GU(B, \sigma)$, we
obtain the group $\U(B, \sigma)$; and by the
argument in \cite[bottom of page 346]{BofInv}, the
induced map $\U(B, \sigma) \to \PGU(B, \sigma)$ is surjective. Its
kernel, being the intersection of $R_{E/F} \mbb G_m$ with $\U(B,
\sigma)$, is equal to $R^1_{E/F} \mbb G_m$, consisting of the 
$\tau$-norm $1$ elements of $R_{E/F} \mbb G_m$.  This gives
us a short exact sequence
\[1 \to R^1_{E/F} \mbb G_m \to \U(B, \sigma) \to \PGU(B, \sigma) \to 1,\]
and hence the desired sequence by base change to $F_\wp$.

Now since $E_\wp =
F_\wp \times F_\wp$, we have $R^1_{E_\wp/F_\wp} \mbb G_m \cong \mbb G_m$,
via $(a,a^{-1}) \mapsto a$.
Hilbert's Theorem~90 together with the six-term cohomology sequence associated to the exact sequence $(*)$ tells us that the map $\oper{U}(B, \tau) \to \PGU(B,
\tau)$ is surjective on $F_\wp$-points.  So the map
$\Sha_{\mc P}(F, \oper{U}(B, \tau)) \to \Sha_{\mc P}(F, \PGU(B, \tau))$ is surjective by
Corollaries~\ref{abstract_Sha_quotient}
and~\ref{patching_equivalence}.  But the rational group $\oper{U}(B, \tau)$ is connected (\cite[VI.23.A]{BofInv}); and so
$\Sha_{\mc P}(F, \oper{U}(B, \tau))$ is trivial by Corollary~\ref{Sha_P-graph} (or by Theorem~3.7 of \cite{HHK}).  Hence so is $\Sha_{\mc P}(F,\PGU(B, \tau))$.

(\ref{unitary_points})
The class of $(B',\tau')$ defines an element of
$H^1(F, \Aut(B, \tau))$ that becomes trivial over each $F_P$.  Thus the corresponding $\Aut(B, \tau)$-torsor has an $F_P$-point for each $P \in X$.  In particular, taking $P$ to be the generic point $\eta_i$ of an irreducible component $X_i$ of $X$, Proposition~\ref{approx} implies that there is an affine dense open subset $U_i$ of $X_i$ that does not meet
any other component of $X$ and such that the torsor has an $F_{U_i}$-point.
Let $\mc U$ be the collection of these sets $U_i$, and let $\mc P$ be the complement of their union.  Then the torsor has an $F_\xi$-point for every $\xi \in \mc P \cup \mc U$; and hence $(B',\tau')$ becomes isomorphic to $(B,\tau)$ over each $F_\xi$.  The conclusion now follows from part~(\ref{unitary_patches}).
\end{proof}

\renewcommand{\labelenumi}{(\alph{enumi})}

\bigskip

\noindent {\bf Author Information:}\\

\noindent David Harbater\\
Department of Mathematics, University of Pennsylvania, Philadelphia, PA 19104-6395, USA\\
email: harbater@math.upenn.edu

\medskip

\noindent Julia Hartmann\\
Lehrstuhl f\"ur Mathematik (Algebra), RWTH Aachen University, 52056 Aachen, Germany\\
email: hartmann@matha.rwth-aachen.de\\
Current contact information:\\
Department of Mathematics, University of Pennsylvania, Philadelphia, PA 19104-6395, USA\\
email: hartmann@math.upenn.edu

\medskip

\noindent Daniel Krashen\\
Department of Mathematics, University of Georgia, Athens, GA 30602, USA\\
email: dkrashen@math.uga.edu

\medskip

\noindent The first author was supported in part by NSF grant DMS-0901164. 
The second author was supported by the German Excellence Initiative via RWTH Aachen University and by the German National Science Foundation (DFG).
The third author was supported in part by NSA grant H98130-08-0109
and NSF grant DMS-1007462.

\end{document}